\documentclass[12pt]{amsart}

\usepackage{amsmath,amssymb,caption,kotex}
\usepackage{tikz-cd,tkz-euclide}
\usepackage[hyphens]{url}
\usepackage[shortlabels]{enumitem}
\usepackage{makecell,longtable,hhline,multirow,array}
\usepackage{hyperref,cleveref}
\newcolumntype{F}[1]{>{\raggedright\arraybackslash\hspace{0pt}}p{#1}}

\theoremstyle{plain}
\newtheorem{thm}{Theorem}[section]
\newtheorem{prop}[thm]{Proposition}
\newtheorem{lem}[thm]{Lemma}
\newtheorem{cor}[thm]{Corollary}
\newtheorem{qtn}[thm]{Question}
\newtheorem{conj}[thm]{Conjecture}
\theoremstyle{definition}
\newtheorem{exm}[thm]{Example}
\newtheorem{dfn}[thm]{Definition}
\theoremstyle{remark}
\newtheorem{clm}[thm]{Claim}
\numberwithin{equation}{section}

\calclayout

\newcommand\scalemath[2]{\scalebox{#1}{\mbox{\ensuremath{\displaystyle #2}}}}

\title[Hyperbolic Dehn Filling, Volume, and Transcendentality]{Hyperbolic Dehn Filling, Volume, and Transcendentality}
\author[B. Jeon]{\bfseries Bogwang Jeon}
\address{Department of Mathematics, POSTECH\\
  77 Cheongam-Ro, Pohang 37673, Republic of Korea}
\email{bogwang.jeon@postech.ac.kr}
\author[S. Oh]{\bfseries Sunul Oh}
\address{Department of Mathematics, POSTECH\\
  77 Cheongam-Ro, Pohang 37673, Republic of Korea}
\email{soonool@postech.ac.kr}


\begin{document}

\begin{abstract}
Let $M$ be a 1-cusped hyperbolic 3-manifold. In this paper, we study the behavior of $N_M(v)$, the number of Dehn fillings of $M$ with a given volume $v(\in \mathbb{R})$. We conduct extensive computational experiments to estimate $N_M$ and propose a theoretical framework to explain its behavior. Further, we prove that the growth of $N_M$ is slower than any power of its filling coefficient.
\end{abstract}

\maketitle

\section{Introduction}\label{intro}

The following is a well-known result of Thurston. 
\begin{thm}[Thurston]\label{T}
Let $M$ be a $1$-cusped hyperbolic 3-manifold and $M(p, q)$ be its $(p, q)$-Dehn filling. Then $M(p, q)$ is hyperbolic for sufficiently large $|p|+|q|$ and 
\begin{equation*}
\textnormal{vol}\,M(p, q)\longrightarrow \textnormal{vol}\,M
\end{equation*} 
as $|p|+|q|$ goes to $\infty$. 
\end{thm}
Throughout this paper, a \textit{hyperbolic $3$-manifold} refers to an orientable, complete hyperbolic $3$-manifold of finite volume.

As a corollary of the above theorem, it follows that 
\begin{cor}
Let $M$ be a $1$-cusped hyperbolic 3-manifold. There are only a finite number of Dehn fillings of $M$ with the same volume. 
\end{cor}

For $M$ as above, if the number of Dehn fillings of $M$ with a given volume $v \in \mathbb{R}$ is denoted by $N_M(v)$, the natural question then to ask is 

\begin{qtn}\cite{G, HM, J1}\label{MainQ}
For a given $1$-cusped hyperbolic 3-manifold $M$, does there exist $c=c(M)$ such that $N_M(v)<c$ for any $v$? 
\end{qtn}

Or, more fundamentally, one may ask 

\begin{qtn}\label{MainQ2}
Let $M$ be a $1$-cusped hyperbolic 3-manifold. Can we completely classify Dehn fillings of $M$ via volume?
\end{qtn}

In \cite{J2}, the first author answered Question \ref{MainQ} affirmatively in a special case as follows:

\begin{thm}[\cite{J2}, Theorem 1.4]\label{realquad}
Let $M$ be a 1-cusped hyperbolic 3-manifold whose cusp shape is quadratic. Then there exists $c$ such that $N_M(v)<c$ for all $v$.
\end{thm}

More quantitatively, the following conjecture was proposed in \cite{J2}:
\begin{conj}
Let $M$ be a $1$-cusped hyperbolic $3$-manifold. Then there exists a finite subgroup $G$ of $\textnormal{GL}_2(\mathbb{Q})$ such that if 
\begin{equation*}
\textnormal{vol}\,M(p',q')=\textnormal{vol}\,M(p,q)
\end{equation*}
where $M(p',q')$ and $M(p,q)$ are two Dehn fillings of $M$, then 
\begin{equation*}
\begin{pmatrix} p' \\ q' \end{pmatrix}=\sigma \begin{pmatrix} p \\ q \end{pmatrix}
\end{equation*}
for some $\sigma \in G$.  
\end{conj}

In this paper, we present experimental evidence supporting the above conjecture and offer a sharper, more precise formulation of it. Along the way, we also establish several theoretical findings, both conditional and unconditional, that either support the conjecture or connect it to other known results.

\subsection{Experimental results}

There is a census compiled by \cite{HW}, \cite{CHW}, \cite{Thi}, and \cite{Bur}, incorporated into \texttt{SnapPy} \cite{CDGW}, a software application designed to analyze the geometry and topology of 3-manifolds. The census consists of 59,107 1-cusped hyperbolic 3-manifolds, obtained by gluing 9 or fewer tetrahedra.

We compute and compare the volumes of $M(p,q)$ for all census manifolds $M$ and all pairs $(p,q)$ with $|p|, |q|\leq100$. To make the numerical estimation of $N_M$ easier, define
\begin{align*}
    N^{100}_M(v)&:=\left|\left\{\pm(p,q):\textnormal{vol}\,M(p,q)=v \text{ and } |p|,|q|\leq100\right\}\right|,\\
    \tilde{N}^{100}_M&:=\max\{n:N^{100}_M(v)=n \text{ for two or more values of $v$}\}.
\end{align*}

For example, the manifold m015\footnote{Here we use SnapPy's notation; see Section \ref{exp} for its detailed explanation.} satisfies $\textnormal{vol\,m015}(-3,2)=\textnormal{vol\,m015}(7,1)$ for this unique pair of fillings within the given range. Hence, we define $\tilde{N}^{100}_{\textnormal{m015}}=1$. The results are summarized in the table below.

\begin{table}[htb]
    \centering
    \begin{tabular}{c|c}
        $\tilde{N}^{100}_M$ & Number of Manifolds \\ \hline\hline
        6 & 1 \\
        4 & 42 \\
        2 & 171 \\
        1 & 58,893 \\ \hline
        Total & 59,107
    \end{tabular}
    \caption{Distribution of $\tilde{N}^{100}_M$ values for census manifolds.}
    \label{disNM}
\end{table}

From the table, we observe that, generically - with more than $99.5\%$ probability - a randomly chosen $1$-cusped hyperbolic $3$-manifold $M$ has its Dehn filling uniquely determined by its volume. However, there are clear exceptions to this phenomenon, which we aim to analyze and for which we seek to provide a theoretical framework.

Before proceeding to the next subsection, we briefly review prior research related to the paper. In \cite{BPZ}, Betley,  Przytycki, and Zukowski showed that the manifold m009 satisfies $\textnormal{vol m009}(p,q)=\textnormal{vol m009}(-p, q)$ for any coprime pair $(p,q)$. In \cite{HMW}, Hodgson, Meyerhoff, and Weeks identified infinitely many pairs of 1-cusped hyperbolic 3-manifolds, where each pair consists of manifolds with the same (complex) volume, obtained by Dehn filling on one cusp of the Whitehead link complement. In \cite{NY}, Neumann and Yang found a manifold, v3066, that generates an infinite family of four Dehn fillings with the same volume. See also \cite{CGHN}. We will revisit these examples in Sections \ref{exp} and \ref{exm}.

\subsection{Theoretical predictions \& confirmations}\label{25033102}

To interpret the experimental results, we first define the \textit{complex volume} of $M$ as
\[\textnormal{vol}_{\mathbb{C}}\,M := \textnormal{vol}\, M + i \textnormal{CS}\, M\quad \mod i\pi^2\mathbb{Z}\]
where CS represents the \textit{Chern-Simons} invariant \cite{NZ, T2, Y}. One may then ask questions analogous to Questions \ref{MainQ}–\ref{MainQ2} for the complex volume. In particular, it is natural to explore how the Chern-Simons invariants behave for two Dehn fillings having the same (real) volume. 

Having closely examined all 214 census manifolds with $\tilde{N}^{100}_{M} >1$ in Table \ref{disNM}, we find that all of them appear to support the following conjecture:

\begin{conj}\label{25030605}
Let $M$ be a $1$-cusped hyperbolic $3$-manifold. If 
$$\textnormal{vol}\, M(p',q')= \textnormal{vol}\, M(p,q)$$
for sufficiently large $|p'|+|q'|$ and $|p|+|q|$, then there exists $n\in \mathbb{N}$ such that either 
$$\textnormal{vol}_{\mathbb{C}}\, M(p',q')=\textnormal{vol}_{\mathbb{C}}\, M(p,q)\quad\mod \frac{i\pi^2}{n}\mathbb{Z}$$
or
$$\textnormal{vol}_{\mathbb{C}}\,M-\textnormal{vol}_{\mathbb{C}}\,M{(p',q')}=\overline{\textnormal{vol}_{\mathbb{C}}\,M-\textnormal{vol}_{\mathbb{C}}\,M{(p,q)}}\quad\mod \frac{i\pi^2}{n}\mathbb{Z}$$
holds.
\end{conj}

To explain the numerical values of $\tilde{N}^{100}_{M}$ appearing Table \ref{disNM} in light of Conjecture \ref{25030605}, let $\mathcal{X}_{ana}$ denote the analytic holonomy set of $M$, which is an analytic curve in $\mathbb{C}^2(:=(u, v))$ parametrizing the deformation space of $M$. Its precise definition will be given in Section \ref{Holvar}. For $\sigma\in\text{\normalfont GL}_2(\mathbb{Q})$, we consider $\sigma$ as an element of $\textnormal{Aut}\,\mathcal{X}_{ana}$, the automorphism group of $\mathcal{X}_{ana}$, if the following map
\begin{equation*}
\begin{aligned}
\mathcal{X}_{ana}&\longrightarrow \mathcal{X}_{ana}\\
\begin{pmatrix} u \\ v \end{pmatrix}&\mapsto \sigma\begin{pmatrix} u \\ v \end{pmatrix}
\end{aligned}
\end{equation*}
is an isomorphism. Under this notation, the following theorem is induced:

\begin{thm}\label{25031903}
Let $M$ be a $1$-cusped hyperbolic $3$-manifold and $\mathcal{X}_{ana}$ be its analytic holonomy set. Let $M(p',q')$ and $M(p,q)$ be two Dehn fillings of $M$ with sufficiently large $|p'|+|q'|$ and $|p|+|q|$. Then the following two statements are equivalent. 
\begin{enumerate}
\item $\textnormal{vol}_{\mathbb{C}}\,M(p',q')=\textnormal{vol}_{\mathbb{C}}\,M(p,q)$ modulo $\frac{i\pi^2}{n} \mathbb{Z}$ where $n\in \mathbb{N}$;  
\item There exists $\sigma:\mathcal{X}_{ana}\xrightarrow{\cong} \mathcal{X}_{ana}$ such that the Dehn filling point of $M(p',q')$ on $\mathcal{X}_{ana}$ corresponds under $\sigma$ to the Dehn filling point of $M(p,q)$ on $\mathcal{X}_{ana}$. Moreover, $\begin{pmatrix}p'\\q'\end{pmatrix}=(\sigma^T)^{-1}  \begin{pmatrix}p\\q\end{pmatrix}$.
\end{enumerate}
In particular, the automorphism group of $\mathcal{X}_{ana}$, consisting of elements arising from the second statement above, is a cyclic group of order $6$ (resp. $4$) when the cusp shape $c_1\in \mathbb{Q}(\sqrt{-3})$ (resp. $\mathbb{Q}(\sqrt{-1})$), and of order $2$ otherwise.  
\end{thm}
Note that, among the elements of $\textnormal{Aut}\,\mathcal{X}_{ana}$ mentioned above, half differ only by a sign. Since we count $M(p,q)$ and $M(-p,-q)$ as the same, one immediate corollary of the previous theorem is 
\begin{cor}\label{25041801}
Let $M$ be a $1$-cusped hyperbolic $3$-manifold and $c_1$ be its cusp shape. If $N_M^{\mathbb{C}}(v)$ denotes the number of hyperbolic Dehn fillings of $M$ with a complex volume $v\in \mathbb{C}$ (modulo $i\pi^2\mathbb{Z}$), then 
\begin{align*}
    N^{\mathbb{C}}_M(v)\leq\begin{cases}
            3 & \text{if $c_1\in \mathbb{Q}(\sqrt{-3})$},\\
            2 & \text{if $c_1\in \mathbb{Q}(\sqrt{-1})$},\\
            1 & \text{otherwise}
         \end{cases}
\end{align*}
for $v$ sufficiently close to $\textnormal{vol}_{\mathbb{C}}\,M$ (modulo $i\pi^2\mathbb{Z}$). 
\end{cor}
In Section \ref{exm}, we provide examples where equality holds for each case in Corollary \ref{25041801}.   

On the other hand, to incorporate with the second conjugate case in Conjecture \ref{25030605}, we define $\overline{\mathcal{X}}_{ana}$ to be the conjugate set of $\mathcal{X}_{ana}$. That is, for $\mathcal{X}_{ana}$ given as
\begin{equation*}
v=\sum_{i=1}^{\infty}c_i u^i, 
\end{equation*}
$\overline{\mathcal{X}}_{ana}\subset\mathbb{C}^2(:=(\overline{u}, \overline{v}))$ is defined as 
\begin{equation}\label{25042801}
\overline{v}=\sum_{i=1}^{\infty}\overline{c}_i \overline{u}^i 
\end{equation}
where $\overline{c_i}$ is the complex conjugate of $c_i$. Indeed, both $\mathcal{X}_{ana}$ and $\overline{\mathcal{X}}_{ana}$ arise from two distinct branches of the same algebraic curve, the holonomy variety of $M$. See Section \ref{Holvar}. 

Our second main theorem of the paper is the following:

\begin{thm}\label{25041802}
Let $M,M(p',q'),M(p,q)$, and $\mathcal{X}_{ana}$ be the same as in Theorem \ref{25031903}. Then the following two statements are equivalent. 
\begin{enumerate}
\item $\textnormal{vol}_{\mathbb{C}}\,M-\textnormal{vol}_{\mathbb{C}}\,M{(p',q')}=\overline{\textnormal{vol}_{\mathbb{C}}\,M-\textnormal{vol}_{\mathbb{C}}\,M{(p,q)}}$ modulo $\frac{i\pi^2}{n}\mathbb{Z}$;

\item There exists $\sigma :\mathcal{X}_{ana}\xrightarrow{\cong} \overline{\mathcal{X}}_{ana}$ such that the Dehn filling point of $M(p',q')$ on $\mathcal{X}_{ana}$ corresponds under $\sigma$ to the Dehn filling point of $M(p,q)$ on $\overline{\mathcal{X}}_{ana}$. Moreover, $\begin{pmatrix}p'\\q'\end{pmatrix}=(\sigma^T)^{-1}  \begin{pmatrix}p\\q\end{pmatrix}$.
\end{enumerate}
\end{thm}

If $\sigma_1, \sigma_2 (\in \textnormal{GL}_2(\mathbb{Q}))$ are isomorphisms from $\mathcal{X}_{ana}$ to $\overline{\mathcal{X}}_{ana}$, then $\sigma_2^{-1}\circ\sigma_1$ is an element of $\textnormal{Aut}\,\mathcal{X}_{ana}$. Hence the following statement is attained from the final observation in Theorem \ref{25031903}:

\begin{cor}\label{ano3}
Let $M, \mathcal{X}_{ana}$, and $\overline{\mathcal{X}}_{ana}$ be the same as before. If the cusp of $M$ is contained in $\mathbb{Q}(\sqrt{-3})$ (resp. $\mathbb{Q}(\sqrt{-1})$), then the number of isomorphisms from $\mathcal{X}_{ana}$ to $\overline{\mathcal{X}}_{ana}$ is at most $6$ (resp. $4$); otherwise, at most $2$.  
\end{cor}

In particular, this, together with Conjecture \ref{25030605}, yields 
\begin{cor}\label{25031904}
Let $M$ be a $1$-cusped hyperbolic $3$-manifold and $c_1$ be its cusp shape. Assuming Conjecture \ref{25030605}, we have
\begin{align*}
    N_M(v)\leq\begin{cases}
            6 & \text{if $c_1\in \mathbb{Q}(\sqrt{-3})$},\\
            4 & \text{if $c_1\in \mathbb{Q}(\sqrt{-1})$},\\
            2 & \text{otherwise}
         \end{cases}
\end{align*}
for $v(\in \mathbb{R})$ sufficiently close to $\textnormal{vol}\,M$.
\end{cor}
The above corollary clearly explains the numbers appearing in the left column of Table \ref{disNM}. If $M$ is the manifold m208 or m135, for instance, then $N_M(v)=6$ or $4$ respectively. See also Section \ref{exm}.

\subsection{Pila–Wilkie}

Another possible approach to Question~\ref{MainQ} is number-theoretic in nature. More precisely, in \cite{J2}, the first author asked whether the Pila–Wilkie counting theorem (Theorem~\ref{PWcount}) can be applied to generalize Theorem~\ref{realquad} to other cusp shapes. Below, we establish the following theorem, whose statement is reminiscent of that of Pila–Wilkie.

\begin{thm}\label{slow}
Let $M$ be a 1-cusped hyperbolic 3-manifold whose cusp shape is not quadratic. For any $\varepsilon>0$, there exists a constant $c>0$ such that
\begin{align*}
N_M(\textnormal{vol}\,M(p,q))\leq c(|p|+|q|)^\varepsilon
\end{align*}
for all $(p,q)\in\mathbb{Z}^2$.
\end{thm}

According to this theorem, if $N_M(\textnormal{vol} \,M(p,q))$ is unbounded as a function of $p$ and $q$, then its growth rate is slower than any power of $|p|+|q|$, which strongly indicates that $N_M$ is actually a bounded function. 

At first glance, the theorem may appear to follow immediately from that of Pila–Wilkie. However, the situation is more subtle, and there are several nontrivial points (e.g., the transcendence of the Neumann–Zagier volume function) that  must be clarified in order to apply the theorem of Pila–Wilkie. We will elaborate more on this in Section \ref{SecSlow}.

\subsection{Application}\label{22071001}

There are a couple of ways in which our results can be applied to other classical problems in the study of hyperbolic 3-manifolds.

As a first application, we obtain
\begin{thm}\label{25050303}
Let $M$ be a $1$-cusped hyperbolic $3$-manifold and $\mathcal{X}$ be its holonomy variety. Suppose $\mathcal{X}$ is rational. For $n\in \mathbb{N}$, if 
\begin{equation*}
\textnormal{vol}_{\mathbb{C}}\,M{(p',q')}=\textnormal{vol}_{\mathbb{C}}\,M{(p,q)}\quad\mod \frac{i\pi^2}{n}\mathbb{Z}
\end{equation*}
with $|p'|+|q'|$ and $|p|+|q|$ sufficiently large, then the two Dehn fillings have the same Bloch invariant.
\end{thm}

The above theorem is an immediate corollary of Theorem \ref{25031903}, combined with Theorem 10.3 of A. Champanerkar in \cite{C}.  

Note that Theorem \ref{25050303} resolves a special case of the following conjecture of D. Ramakrishnan:
\begin{conj}[D. Ramakrishnan]
Let $M$ and $N$ be two hyperbolic $3$-manifolds. If 
\begin{equation*}
\textnormal{vol}_\mathbb{C}\,M=\textnormal{vol}_\mathbb{C}\,N\quad\mod i\pi^2\mathbb{Q}, 
\end{equation*}
then both $M$ and $N$ have the same Bloch invariant. 
\end{conj}

On the other hand, the orientation reversing analogue of Theorem \ref{25050303} is stated as follows:
\begin{thm}\label{25050304}
Let $M$ be a $1$-cusped hyperbolic $3$-manifold and $\mathcal{X}$ be its holonomy variety. Suppose $\mathcal{X}$ is rational. For $n\in \mathbb{N}$, if 
\begin{equation*}
\textnormal{vol}_{\mathbb{C}}\,M-\textnormal{vol}_{\mathbb{C}}\,M{(p',q')}=\overline{\textnormal{vol}_{\mathbb{C}}\,M-\textnormal{vol}_{\mathbb{C}}\,M{(p,q)}} \mod \frac{i\pi^2}{n}\mathbb{Z}
\end{equation*}
with $|p'|+|q'|$ and $|p|+|q|$ sufficiently large, then 
\begin{equation*}
\beta(M)-\beta(M{(p',q')})=-\overline{\beta(M)}+\overline{\beta(M{(p,q)})}
\end{equation*}
where $\beta(M)$ denotes the Bloch invariant of $M$ (and similarly for other terms).
\end{thm}

Since the scissors congruence group is a quotient of the Bloch group obtained by identifying each element with its orientation-reversed counterpart, due to the work of Dupont and Sah \cite{DS}, one immediate corollary of Theorems \ref{25050303} and \ref{25050304} is the following:
\begin{cor}\label{25050307}
Let $M$ be a $1$-cusped hyperbolic $3$-manifold and $\mathcal{X}$ be its holonomy variety. Suppose $\mathcal{X}$ is rational. Assuming Conjecture \ref{25030605}, if $M{(p',q')}$ and $M{(p,q)}$ are two Dehn fillings of $M$ having the same volume with $|p'|+|q'|$ and $|p|+|q|$ sufficiently large, then $M{(p',q')}$ is scissors congruent to $M{(p,q)}$. 
\end{cor}

The proof of Theorem \ref{25050304} follows from an elementary observation; however, we do not include it in this paper, as it requires additional background that goes far beyond what is needed for the proofs of the main theorems.

\subsection{Outline of the paper}

The paper is organized as follows. In Section \ref{prel}, we introduce key background concepts, including the holonomy variety and the Neumann-Zagier volume formula. In Section \ref{anosub}, we prove our first two main theorems, Theorems \ref{25031903} and \ref{25041802}. In Section \ref{exp}, we present numerical experiment data along with some theoretical interpretation. Sections \ref{SecTr}-\ref{SecSlow} are devoted to establishing our last main result, Theorem \ref{slow}: in Section \ref{SecTr}, we carry out preparatory work for the proof and, building on this, verify the theorem in Section \ref{SecSlow}. Finally, in Section \ref{exm}, we provide concrete examples with detailed proofs to illustrate the optimality of our theorems.

\subsection*{Acknowledgements}

The first author would like to thank Abhijit Champanerkar for an enlightening conversation, in particular for sharing his motivation and research background from the early 2000s, when his thesis, \cite{C}, was written. 

The authors would like to thank Seewoo Lee for allowng the use of his proof of Lemma \ref{numb1}. 

\section{Preliminaries}\label{prel}

\subsection{Holonomy variety: algebraic \& analytic }\label{Holvar}

Let $M$ be a 1-cusped hyperbolic 3-manifold, and $\mathcal{T}$ be an ideal triangulation of $M$ with shape parameters $z_1, \cdots, z_n$ assigned to its ideal tetrahedra. The \textit{gluing variety} $G(\mathcal{T})(\subset \mathbb{C}^n)$ of $M$, defined by equations in $z_1, \cdots, z_n$, encodes the conditions under which the $n$ tetrahedra in $\mathcal{T}$ are glued together along their edges to yield a (possibly incomplete) hyperbolic structure on $M$. It is roughly regarded as the space of all such structures, with each point in $G(\mathcal{T})$ corresponding to a $\text{PSL}_2(\mathbb{C})$ representation of $\pi_1(M)$. 

Let $T^2$ be a torus cross-section of the cusp and $(\mathbf{m}, \mathbf{l})$ be a chosen basis of $H_1(T^2)$. Let $m$ and $l$ be the squares of the eigenvalues of the images of $\mathbf{m}$ and $\mathbf{l}$ under the representation corresponding to a point in $G(\mathcal{T})$. Then $\mathbf{m}$ and $\mathbf{l}$ are expressed as rational functions on $G(\mathcal{T})$, which leads to the following natural map: $$\mathrm{Hol}:G(\mathcal{T})\longrightarrow(m,l).$$

The \textit{holonomy variety} $\mathcal{X}$ of $M$ is then defined as the Zariski closure of the image of the map Hol. It is known that $\mathcal{X}$ is defined over $\mathbb{Q}$ and does not depend on the choice of triangulation $\mathcal{T}$. For further details, we refer the reader to \cite{CCGLS} or \cite{C}.

Among the irreducible components of the holonomy variety, there exists a distinguished one, that is, the \textit{geometric component} of it. This component is an algebraic curve in $\mathbb{C}^2$ with coordinates $(m,l)$ and contains $(1,1)$, which corresponds to the complete hyperbolic structure of $M$.

Now let $u:=\log m$ and $v:=\log l$. There is a neighborhood $\mathcal{N}_\mathcal{X}$ of a branch of $\mathcal{X}$ containing $(1,1)$, that is locally biholomorphic to a complex manifold containing $(0,0)$ in $\mathbb{C}^2$ parametrized by $u$ and $v$. We term this the \textit{analytic holonomy set} and represent it by $\mathcal{X}_{ana}$. It is well-known that $v$ can be expressed as an odd analytic function of $u$, given by
\[v=\sum_{i=1}^\infty c_iu^i\quad(c_i=0\text{ for even $i$}).\]
Here $c_1$ is the \textit{cusp shape} of $M$ with respect to $\mathbf{m}, \mathbf{l}$, which lies in the upper half plane. Note that, since $\mathcal{X}$ is defined over $\mathbb{Q}$, the conjugate set $\overline{\mathcal{X}}_{ana}$ of $\mathcal{X}_{ana}$, given in \eqref{25042801}, is also biholomorphic to a branch of $\mathcal{X}$ containing $(1,1)$. 

For $(u,v)\in\mathcal{X}_{ana}$ near $(0,0)$, there exists the unique real pair $(p,q)\in\mathbb{R}^2$ satisfying $pu+qv=2\pi i$. The correspondence $u\mapsto (p,q)$ is a homeomorphism from a neighborhood of 0 in $\mathbb{C}$ onto a neighborhood of $\infty$ in the one point compactification $\mathbb{R}^2\cup\{\infty\}$. We call the pair $(p,q)$ the \textit{generalized Dehn filling coefficient} (See \cite{NZ} \S1 or \cite{NR} 6.2).

\subsection{Neumann-Zagier volume formula}\label{NZvf}

If a point of $\mathcal{X}_{ana}$ corresponds to an incomplete hyperbolic structure on $M$ and has a coprime integer pair $(p,q)$ as its generalized Dehn filling coefficient, then its completion is known to be isometric to the $(p,q)$-Dehn filling $M(p,q)$. We term such a point the $(p,q)$-\textit{Dehn filling point} on $\mathcal{X}_{ana}$ and the set of limit points added under the completion, which is known to form a geodesic, the \textit{core geodesic} $\boldsymbol{\gamma}{(p,q)}$ of $M(p,q)$. 

The volume of $M(p,q)$ is computed via the following expression, referred to as the \textit{Neumann-Zagier volume formula}:
\begin{align*}
\textnormal{vol} \,M(p,q)=\textnormal{vol}\, M&-\text{Im}(c_1)\frac{\pi^2}{|A|^2}+\frac{1}{4}\text{Im}\Biggl[\frac{1}{2}c_3\left(\frac{2\pi}{A}\right)^4+\frac{1}{3}\left(3c_3^2\frac{q}{A}-c_5\right)\left(\frac{2\pi}{A}\right)^6\cdots\Biggr]
\end{align*}
where $A:=p+c_1q$. To simplify notation, we set\footnote{We will often omit the subscript $M$ in $\Theta_M$ when the context makes it clear.}
 
$$\Theta_M(p,q):=\textnormal{vol} \,M(p,q)-\textnormal{vol} \,M.$$ 

Changing the basis $\mathbf{m}, \mathbf{l}$ of $H_1(T^2)$ by $$\tilde{\mathbf{m}}:=\mathbf{m}^a\mathbf{l}^b, \qquad\tilde{\mathbf{l}}:=\mathbf{m}^c\mathbf{l}^d$$ where 
$\begin{pmatrix}a & b\\ c & d\end{pmatrix}\in \text{SL}_2(\mathbb{Z})$, and letting 
$$\tilde{u}:=au+bv, \quad \tilde{v}:=cu+dv,$$ it follows that 
$$\tilde{v}=\tilde{c_1}\tilde{u}+\tilde{c_3}\tilde{u}^3+\cdots$$ where
\begin{align*}
    \tilde{c_1}=\frac{c+dc_1}{a+bc_1}, \quad\tilde{c_3}=\frac{c_3}{(a+bc_1)^4},\quad \cdots.
\end{align*}
The $(p,q)$-Dehn filling with respect to $\mathbf{m}, \mathbf{l}$ then becomes the $(\tilde{p},\tilde{q}):=(dp-cq,-bp+aq)$-Dehn filling with respect to $\tilde{\mathbf{m}}, \tilde{\mathbf{l}}$, and $\tilde{\Theta}(\tilde{x},\tilde{y}):=\textnormal{vol} \,M{(\tilde{p},\tilde{q})}-\textnormal{vol} \,M$ is given by 
\begin{align}
    \tilde{\Theta}(\tilde{p},\tilde{q})&=-\text{Im}(\tilde{c_1})\frac{\pi^2}{|\tilde{A}|^2}+\frac{1}{4}\text{Im}\Biggl[\frac{1}{2}\tilde{c_3}\left(\frac{2\pi}{\tilde{A}}\right)^4+\frac{1}{3}\left(3\tilde{c_3}^2\frac{\tilde{q}}{\tilde{A}}-\tilde{c_5}\right)\left(\frac{2\pi}{\tilde{A}}\right)^6\cdots\Biggr]\label{Thetatilde}
\end{align}
where $\tilde{A}:=\tilde{p}+\tilde{c_1}\tilde{q}$. This basis change will play a key role in the proofs of Theorems \ref{main2}, \ref{main3}, and \ref{transc2}. 

\subsection{Complex length and volume}

The \textit{complex length} $\ell(\boldsymbol{\gamma})$ of a geodesic $\boldsymbol{\gamma}$ is defined to be $length(\boldsymbol{\gamma})+i\cdot torsion(\boldsymbol{\gamma})$. For the $(p,q)$-Dehn filling of $M$, if we let $\boldsymbol{\gamma}=\boldsymbol{\gamma}{(p,q)}$, then it is computed via the following lemma. 

\begin{lem}[\cite{NZ}, Lemma 4.2] \label{lemNZ}
Let $r$ and $s$ be integers satisfying $\det \begin{pmatrix}p & q \\r &s\end{pmatrix}=1$. Then $\ell(\boldsymbol{\gamma})=-(ru+sv)$ modulo $2\pi i\mathbb{Z}$.
\end{lem}

Since $ps-qr=1$ and $pu+qv=2\pi i$, it follows that
\begin{align*}
    \ell(\boldsymbol{\gamma})=-\frac{v}{p}-\frac{2\pi i r}{p}=\frac{u}{q}-\frac{2\pi i s}{q}\quad\mod 2\pi i\mathbb{Z}.
\end{align*}

By \cite{Y} (See also \cite{N} below (3.6)), there exists a complex analytic function $f(u)$ near $u=0$ such that $f(0)=0$ and
\begin{align}
f(u)=\textnormal{vol}_{\mathbb{C}}\,M(p,q)-\text{vol}_{\mathbb{C}}\,M+\frac{\pi}{2}\ell(\boldsymbol{\gamma})\quad\mod i\pi^2\mathbb{Z}.\label{Nzf1}
\end{align}
Note that the function $f$ in the notation of \cite{NZ} is $i$ times $f$ here. By (61) of \cite{NZ},
\begin{align}
f(u)=\frac{i}{4}\left[
\frac{1}{2}c_3\left(\frac{2\pi}{A}\right)^4+\frac{2}{3}\left(3c_3^2\frac{q}{A}-c_5\right)\left(\frac{2\pi}{A}\right)^6+\cdots\right],\label{Nzf2}
\end{align}
hence $|f(u)|=O(|A|^{-4})$ for sufficiently large $|p|+|q|$. Consequently, the complex length $\ell(\boldsymbol{\gamma})$, up to scaling, is the majority part of the complex volume change.

From (58) of \cite{NZ}, we have
\begin{align*}
\ell(\boldsymbol{\gamma})=\frac{u}{q}-\frac{2\pi i s}{q}=\frac{2\pi i}{qA}-\frac{2\pi i s}{q}-\frac{c_3}{A}\left(\frac{2\pi i}{A}\right)^3+\cdots\quad\mod 2\pi i\mathbb{Z}
\end{align*}
for sufficiently large $|p|+|q|$, and, as a result,
\begin{align*}
\textnormal{vol}_{\mathbb{C}}\,M(p,q)=\textnormal{vol}_{\mathbb{C}}\,M-\frac{i\pi^2 }{qA}+\frac{i\pi^2 s}{q}+\cdots\quad\mod i\pi^2 \mathbb{Z},
\end{align*}
and
\begin{align}
\left|-\textnormal{vol}_{\mathbb{C}}\,M(p,q)+\textnormal{vol}_{\mathbb{C}}\,M-\frac{i\pi^2 }{qA}+\frac{i\pi^2 s}{q}\right|=O(|A|^{-4})\label{cvol}
\end{align}
for some suitable choice of complex volume with sufficiently large $|p|+|q|$.

The following lemma is straightforward.

\begin{lem}\label{Anear}
Having the same notation above, let $A':=p'+c_1q'$. If $\textnormal{vol}\,M(p,q)= \textnormal{vol}\,M{(p',q')}$ with sufficiently large $|p|+|q|$, then $O(|A|)=O(|A'|)$ and $O(|p|+|q|)=O(|p'|+|q'|)$.
\end{lem}

\section{Proofs of main I and II}\label{anosub}

In this section, we verify our first two main theorems, Theorems \ref{25031903} and \ref{25041802}.

\subsection{Proof of Theorem \ref{25031903}}
Before proving Theorem \ref{25031903}, we first state the following theorem regarding the relation between the complex volume and the complex length of the core geodesic. 

\begin{thm}\label{main2}
Let $M$ be a 1-cusped hyperbolic 3-manifold. Given $n\in\mathbb{N}$,
\begin{equation}\label{CVrational}
\textnormal{vol}_{\mathbb{C}}\,M{(p,q)}=\textnormal{vol}_{\mathbb{C}}\,M{(p',q')} 
\quad\mod\frac{i\pi^2}{n}\mathbb{Z}
\end{equation}
is equivalent to
$$\ell(\boldsymbol{\gamma}(p,q))=\ell(\boldsymbol{\gamma}(p',q'))\quad  \mod\frac{2\pi i}{n}\mathbb{Z}$$
for sufficiently large $|p|+|q|$.
\end{thm}

To prove the theorem, we use the hyperbolic metric on $\mathbb{H}^2=\{z\in\mathbb{C}:\textnormal{Im}(z)>0\}=\{(x,y)\in\mathbb{R}^2:y>0\}$, given by $ds^2=\frac{dx^2+dy^2}{y}$. The following lemma is basic. 

\begin{lem}\label{2dimhyp}
Let $z_0=x_0+iy_0\in\mathbb{H}^2$ and $r>0$.
\begin{enumerate}
\item Consider the Euclidean circle $\{z\in\mathbb{H}^2:|z-z_0|=r\}$. Then, with respect to the hyperbolic metric, the farthest point from $z_0$ on the circle is $z_0-ri$, and the distance is given by $\textnormal{d}_{\mathbb{H}^2}(z_0,z_0-ri)=\log\frac{y_0}{y_0-r}$.
\item Consider the hyperbolic circle $\{z\in\mathbb{H}^2:\textnormal{d}_{\mathbb{H}^2}(z,z_0)=r\}$. Then, with respect to the Euclidean metric, the farthest point from $z_0$ on the circle is $x_0+e^ry_0i$, and the distance is given by $y_0(e^r-1)$.
\end{enumerate}
\end{lem}

Recall that the isometries of $\mathbb{H}^2$ are given by M\"obius transformations. The next lemma concerns the rigidity of these elements when the coefficients lie in $\tfrac{1}{n}\mathbb{Z}$ under a certain condition.

\begin{lem}\label{matrixfinite}
Let $c_1\in\mathbb{H}^2$ and $(p,q)$ be a coprime integer pair. Fix $n\in\mathbb{N}$ and let $\xi_{(p,q)}(z)=\frac{\alpha z+\beta}{\gamma z +\delta}$ for $\alpha,\beta,\gamma,\delta\in\frac{1}{n}\mathbb{Z}$ satisfying $\alpha\delta-\beta\gamma=1$. Then there are only finitely many choices for $\begin{pmatrix}\alpha & \beta \\ \gamma &\delta\end{pmatrix}$ satisfying
\begin{align}\label{25050101}
\textnormal{Im}(c_1-\xi_{(p,q)}(c_1))=O((|p|+|q|)^{-2})\quad\textnormal{and}\quad\textnormal{Re}(\xi_{(p,q)}(c_1))=O(1)
\end{align}
with sufficiently large $|p|+|q|$.
\end{lem}

\begin{proof}
By a simple calculation, 
\begin{align*}
\textnormal{Im}(\xi_{(p,q)}(c_1))=\textnormal{Im}\left(\frac{\alpha c_1+\beta}{\gamma c_1 +\delta}\right)=\textnormal{Im}\left(\frac{(\alpha c_1+\beta)(\gamma \overline{c_1} +\delta)}{|\gamma c_1 +\delta|^2}\right)=\frac{\textnormal{Im}(c_1)}{|\gamma c_1 +\delta|^2}. 
\end{align*}
Combining this with the first assumption in \eqref{25050101}, we get $1-\frac{1}{|\gamma c_1 +\delta|^2}=O((|p|+|q|)^{-2})$, that is, $|\gamma c_1 +\delta|$ is arbitrarily close to 1. Note that, since $c_1\notin\mathbb{R}$, there are only finitely many choices for $(\gamma, \delta)$ in $\frac{1}{n}(\mathbb{Z}\times\mathbb{Z})$. Further, if $(\alpha_0, \beta_0)\in\frac{1}{n}(\mathbb{Z}\times\mathbb{Z})$ where $\alpha_0\delta-\beta_0\gamma=1$, then any possible choice of $(\alpha, \beta)$ must be of the form $(\alpha_0+k\gamma, \beta_0+k\delta)$ for some $k\in\frac{1}{n\gamma}\mathbb{Z}$ or $\frac{1}{n\delta}\mathbb{Z}$.

Now,
\begin{align*}
\textnormal{Re}(\xi_{(p,q)}(c_1))&=\frac{\alpha\gamma|c_1|^2+(\alpha\delta+\beta\gamma)\textnormal{Re}(c_1)+\beta\delta}{|\gamma c_1 +\delta|^2}\\
&=\frac{\alpha_0\gamma|c_1|^2+(\alpha_0\delta+\beta_0\gamma)\textnormal{Re}(c_1)+\beta_0\delta+k|\gamma c_1+\delta|^2}{|\gamma c_1 +\delta|^2}=O(1).
\end{align*}
Consequently, for each $(\gamma, \delta)$, the number of possible choices of $k$ is finite, which completes the proof.
\end{proof}

\begin{proof}[Proof of Theorem \ref{main2}]
Suppose that \eqref{CVrational} holds for sufficiently large $|p|+|q|$. Let $A=p+c_1q$ and $A'=p'+c_1q'$. From \eqref{cvol} and Lemma \ref{Anear}, we can choose integers $r, s, r'$, and $s'$ satisfying $ps-qr=p's'-q'r'=1$ and
\begin{align*}
\left|-\frac{i\pi^2 }{qA}+\frac{i\pi^2 s}{q}+\frac{i\pi^2}{q'A'}-\frac{i\pi^2 s'}{q'}-\frac{i\pi^2k}{n}\right|=O(|A|^{-4})
\end{align*}
where $0\leq k<n$. 

Define $B=r+c_1s$ and $B'=r'+c_1s'$. Since
\begin{align*}
-\frac{1}{qA}+\frac{ s}{q}=\frac{-1+sA}{qA}=\frac{-(ps-qr)+s(p+c_1q)}{qA}=\frac{B}{A},
\end{align*}
it follows that
\begin{align}\label{25043003}
\left|\frac{B}{A}-\frac{B'}{A'}-\frac{k}{n}\right|=O(|A|^{-4}).
\end{align}

As $\begin{pmatrix}s & r \\q &p\end{pmatrix}\in \text{PSL}_2(\mathbb{Z})$, it acts as a M\"obius transformation $\sigma_{(p,q)}(z):=\frac{sz+r}{qz+p}$ on $\mathbb{H}^2$. If $\tau_{k/n}$ denotes $\begin{pmatrix}1 & k/n \\0 &1\end{pmatrix}$, then $\sigma_{(p,q)}(c_1)=\frac{B}{A}$, and so \eqref{25043003} implies
$$|\sigma_{(p,q)}(c_1)-\tau_{k/n}\circ\sigma_{(p',q')}(c_1)|<C|A|^{-4}$$
for some $C>0$. Since Im$(\sigma_{(p,q)}(c_1))=\textnormal{Im}(c_1)|A|^{-2}$, by Lemma \ref{2dimhyp} (1),
\begin{equation*}
\begin{aligned}
\textnormal{d}_{\mathbb{H}^2}\big(\sigma_{(p,q)}(c_1), \tau_{k/n}\circ\sigma_{(p',q')}(c_1)\big) &\leq \log\frac{\textnormal{Im}(c_1)|A|^{-2}}{\textnormal{Im}(c_1)|A|^{-2}-C|A|^{-4}}\\
\Longrightarrow \textnormal{d}_{\mathbb{H}^2}\big(c_1, \sigma_{(p,q)}^{-1}\circ\tau_{k/n}\circ\sigma_{(p',q')}(c_1)\big) &\leq \log\frac{\textnormal{Im}(c_1)|A|^{-2}}{\textnormal{Im}(c_1)-C|A|^{-2}}, 
\end{aligned}
\end{equation*}
and, by Lemma \ref{2dimhyp} (2), 
\begin{equation}\label{24043005}
|c_1-\sigma_{(p,q)}^{-1}\circ \tau_{k/n}\circ\sigma_{(p',q')}(c_1)|<\textnormal{Im}(c_1)\left(\frac{\textnormal{Im}(c_1)}{\textnormal{Im}(c_1)-C|A|^{-2}}-1\right)=O(|A|^{-2})
\end{equation} 
for sufficiently large $|p|+|q|$. 

By Lemma \ref{matrixfinite}, the number of possible choices for $\sigma_{(p,q)}^{-1}\circ \tau_{k/n}\circ\sigma_{(p',q')}$ is finite, so we assume $\begin{pmatrix}\alpha & \beta \\ \gamma &\delta\end{pmatrix}$ is fixed. Since
\begin{align*}
\begin{pmatrix}s & r \\q &p\end{pmatrix}\begin{pmatrix}\alpha & \beta \\ \gamma &\delta\end{pmatrix}=\begin{pmatrix}1 & k/n \\0 &1\end{pmatrix}\begin{pmatrix}s' & r' \\q' &p'\end{pmatrix},
\end{align*}
it follows that $(p',q')=(\delta p+\beta q,\gamma p+\alpha q).$ Let $(u_0,v_0)$ (resp. $(u_0',v_0')$) be the point of $\mathcal{X}_{ana}$ corresponding to the $(p,q)$ (resp. $(p',q')$)-Dehn filling. As in Section \ref{NZvf}, define $(\tilde{u}, \tilde{v}):=(\delta u+\gamma v,\beta u+\alpha v)$ where $\tilde{v}=\tilde{c_1}\tilde{u}+\tilde{c_3}\tilde{u}^3+\cdots$. Then 
$$\tilde{c_1}=\frac{\beta+\alpha c_1}{\delta+\gamma c_1}=c_1$$
by \eqref{24043005}, and the $(p',q')$-Dehn filling with respect to $\mathbf{m}, \mathbf{l}$ is the $(p,q)$-Dehn filling with respect to the new parametrization. Therefore, 
$$\Theta(p,q)=\Theta(p',q')=\tilde{\Theta}(p,q)$$
where $\Theta(p,q)=\textnormal{vol} \,M{(p,q)}-\textnormal{vol} \,M$ and $\tilde{\Theta}$ is as given in \eqref{Thetatilde}.

Since we can assume these equations hold for infinitely many coefficients $(p,q)$, by comparing the terms of $\Theta$ and $\tilde{\Theta}$, we have $c_j=\tilde{c_j}$ for $j\geq3$. Consequently, by rewriting the equations \eqref{Nzf1} and \eqref{Nzf2} with respect to the new parametrization, we get
\begin{align*}
f(u_0')&=\textnormal{vol}_{\mathbb{C}}\,M{(p',q')}-\text{vol}_{\mathbb{C}}\,M+\frac{\pi}{2}\ell(\boldsymbol{\gamma}(p',q'))&&\mod i\pi^2\mathbb{Z}\\
&=\frac{i}{4}\left[
\frac{1}{2}\tilde{c_3}\left(\frac{2\pi}{\tilde{A}}\right)^4+\frac{2}{3}\left(3\tilde{c_3}^2\frac{q}{\tilde{A}}-\tilde{c_5}\right)\left(\frac{2\pi}{\tilde{A}}\right)^6+\cdots\right]\\
=f(u_0)&=\textnormal{vol}_{\mathbb{C}}\,M{(p,q)}-\text{vol}_{\mathbb{C}}\,M+\frac{\pi}{2}\ell(\boldsymbol{\gamma}(p,q))&&\mod i\pi^2\mathbb{Z}
\end{align*}
where $\tilde{A}=p+\tilde{c_1}q=A$, implying
$$\ell(\boldsymbol{\gamma}(p,q))=\ell(\boldsymbol{\gamma}(p',q'))\mod \frac{2\pi i}{n}\mathbb{Z}.$$

The proof of the converse follows the same argument, as the complex length of the core geodesic and the complex volume share the dominant term up to a scaling factor.
\end{proof}

An immediate consequence of Theorem \ref{main2} is the following, which in turn implies Theorem \ref{25031903} as a corollary:
\begin{cor}\label{maincor}
Let $M$ be a 1-cusped hyperbolic 3-manifold and $\mathcal{X}_{ana}$ be its analytic holonomy set. Let $M(p',q')$ and $M(p,q)$ be two Dehn fillings of $M$ with sufficiently large $|p'|+|q'|$ and $|p|+|q|$. Then the following three statements are equivalent: 
\begin{enumerate}
\item $\textnormal{vol}_{\mathbb{C}}\,M(p',q')=\textnormal{vol}_{\mathbb{C}}\,M(p,q)$ modulo $\frac{i\pi^2}{n}\mathbb{Z}$; 
\item $\ell(\boldsymbol{\gamma}(p,q))=\ell(\boldsymbol{\gamma}(p',q'))$ modulo $\frac{2\pi i}{n}\mathbb{Z}$; 
\item there exists $\sigma:\mathcal{X}_{ana}\xrightarrow{\cong} \mathcal{X}_{ana}$ such that the Dehn filling point of $M(p',q')$ on $\mathcal{X}_{ana}$ corresponds under $\sigma$ to the Dehn filling point of $M(p,q)$ on $\mathcal{X}_{ana}$. 
\end{enumerate}
In particular, if the cusp shape of $M$ is contained in $\mathbb{Q}(\sqrt{-3})$ (resp. $\mathbb{Q}(\sqrt{-1})$), then $\sigma$ has order $6$ (resp. $4$) and $\begin{pmatrix} p' \\ q' \end{pmatrix}=(\sigma^T)^{-k}\begin{pmatrix} p \\ q\end{pmatrix}$ for some $k\in\mathbb{Z}$. Otherwise, $\sigma=\pm Id$. 
\end{cor}

\begin{proof}
The equivalence of (1) and (2) is exactly the statement of Theorem \ref{main2}. By Theorem 4.3 (and its proof) in \cite{J1}, (2) implies (3). 

To prove $(3)\Longrightarrow (2)$, let $\begin{pmatrix}
u_0\\
v_0
\end{pmatrix}$ (resp. $\begin{pmatrix}
u_0'\\
v_0'
\end{pmatrix}$) be the Dehn filling point of $M(p,q)$ (resp. $M(p', q')$) on $\mathcal{X}_{ana}$, and $\sigma:=\begin{pmatrix}
\alpha & \beta\\
\gamma & \delta 
\end{pmatrix}$. Then   
$$\begin{pmatrix}
u_0'\\
v_0'
\end{pmatrix}=
\begin{pmatrix}
\alpha & \beta\\
\gamma & \delta 
\end{pmatrix}
\begin{pmatrix}
u_0\\
v_0
\end{pmatrix}$$
by the assumption, and thus 
\begin{equation}\label{25042102}
\begin{pmatrix}
p' & q'\\
r' & s'
\end{pmatrix}
\begin{pmatrix}
u_0'\\
v_0'
\end{pmatrix}=
\begin{pmatrix}
2\pi i\\
-\ell(\boldsymbol{\gamma}(p',q'))
\end{pmatrix}
=\begin{pmatrix}
p' & q'\\
r' & s'
\end{pmatrix}
\begin{pmatrix}
\alpha & \beta\\
\gamma & \delta 
\end{pmatrix}
\begin{pmatrix}
u_0\\
v_0
\end{pmatrix}
\end{equation}
where $r', s'\in \mathbb{Z}$ and $p's'-q'r'=1$ by Lemma \ref{lemNZ}. Since $pu_0+qv_0=2\pi i$, and $e^{u_0}$ and $e^{v_0}$ are not roots of unity, it follows that 
$$\begin{pmatrix}
p' & q'
\end{pmatrix}
\begin{pmatrix}
\alpha & \beta\\
\gamma & \delta 
\end{pmatrix}=\begin{pmatrix}
p & q
\end{pmatrix},$$ and so, for $r, s\in \mathbb{Z}$ satisfying $ps-qr=1$, 
\begin{equation}\label{25042101}
\begin{pmatrix}
r & s
\end{pmatrix}-\begin{pmatrix}
r' & s'
\end{pmatrix}
\begin{pmatrix}
\alpha & \beta\\
\gamma & \delta 
\end{pmatrix}=m\begin{pmatrix}
p & q
\end{pmatrix}
\end{equation} 
for some $m\in \mathbb{Q}$. Multiplying $\begin{pmatrix}
u_0 \\
v_0
\end{pmatrix}$ to the left-hand sides of \eqref{25042101} and comparing the result with \eqref{25042102}, we conclude 
$$\ell(\boldsymbol{\gamma}(p,q))-\ell(\boldsymbol{\gamma}(p',q'))\in \frac{2\pi i}{n} \mathbb{Z}$$
where $n\in \mathbb{N}$. 

The last statement is also a result of Theorem 4.3 in \cite{J1}. (However, since both the $+$ and $-$ signs are counted as a single case in \cite{J1}, the order that appears there is half of the order considered here.)  
\end{proof}

If $n=1$, then the element $\sigma$ in the third statement of the above corollary clearly lies in $\textnormal{SL}_2(\mathbb{Z})$.

\noindent \textbf{Remark. } Theorem 4.3 in \cite{J1} is a special case of the so-called Zilber-Pink conjecture. Its proof is phrased in slightly different language, involving what is known as an anomalous subset. However, it is essentially equivalent to the version adopted in this paper. Specifically, for $\sigma \in \mathrm{GL}_2(\mathbb{Q})$, consider a subset of $\mathcal{X}_{ana} \times \mathcal{X}_{ana} \subset \mathbb{C}^4 (:= (u, v, u', v'))$ defined by
$$\begin{pmatrix}
u'\\
v'
\end{pmatrix}=\sigma \begin{pmatrix}
u\\
v
\end{pmatrix}.$$
Then the subset is \textit{anomalous} in $\mathcal{X}_{ana}\times \mathcal{X}_{ana}$ if and only if the map
\begin{equation*}
\begin{aligned}
\mathcal{X}_{ana}&\longrightarrow \mathcal{X}_{ana}\\
\begin{pmatrix} u \\ v \end{pmatrix}&\mapsto \sigma\begin{pmatrix} u \\ v 
\end{pmatrix}
\end{aligned}
\end{equation*}
is an isomorphism. 

Conceivably, the statement of Theorem \ref{25031903} (or its analogue) has been known to other researchers, including W. Neumann, since at least the 2000s. More specifically, given two manifolds (not necessarily the same), if they admit infinitely many pairs of Dehn fillings - one from each - with the same volume, it has been believed that their holonomy varieties would coincide.

\subsection{Proof of Theorem \ref{25041802}}
In this subsection, we establish Theorem \ref{25041802}. The general scheme is similar to that presented in the previous subsection. We begin with the following, which can be seen as a counterpart of Theorem \ref{main2}. 

\begin{thm}\label{main3}
Let $M$ be a 1-cusped hyperbolic 3-manifold. Given $n\in\mathbb{N}$,
\begin{align}\label{CVconjurational}
\textnormal{vol}_{\mathbb{C}}\,M-\textnormal{vol}_{\mathbb{C}}\,M{(p,q)}=\overline{\textnormal{vol}_{\mathbb{C}}\,M-\textnormal{vol}_{\mathbb{C}}\,M{(p',q')}}\quad\mod\frac{i\pi^2}{n}\mathbb{Z}
\end{align}
is equivalent to
$$\ell(\boldsymbol{\gamma}(p,q))=\overline{\ell(\boldsymbol{\gamma}(p',q'))}\quad  \mod\frac{2\pi i}{n}\mathbb{Z}$$
for sufficiently large $|p|+|q|$.
\end{thm}

\begin{proof}
As in the previous notation, \eqref{CVconjurational} is equivalent to
\begin{align*}
\left|\frac{B}{A}+\frac{\overline{B'}}{\overline{A'}}+\frac{k}{n}\right|=O(|A|^{-4}).
\end{align*}

Define $\zeta(z)=-\overline{z}$. Then $|\sigma_{(p,q)}(c_1)-\zeta\circ\tau_{-k/n}\circ\sigma_{(p',q')}(c_1)|=O(|A|^{-4})$
and so 
$$|c_1-\sigma_{(p,q)}^{-1}\circ \zeta\circ\tau_{-k/n}\circ\sigma_{(p',q')}(c_1)|=O(|A|^{-2}),$$ 
following a similar procedure shown in the proof of Theorem \ref{main2}.

By letting $\sigma_{(-p,q)}(z):=\frac{-sz+r}{qz-p}$, we have $\sigma_{(p,q)}\circ\zeta=\zeta\circ\sigma_{(-p,q)}$, so 
$$|c_1-\zeta\circ\sigma_{(-p,q)}^{-1}\circ \tau_{-k/n}\circ\sigma_{(p',q')}(c_1)|=O(|A|^{-2}),$$ implying
\begin{align*}
\textnormal{Im}(c_1-\sigma_{(-p,q)}^{-1}\circ \tau_{-k/n}\circ\sigma_{(p',q')}(c_1))&=O(|A|^{-2}),\\
\textnormal{Re}(\sigma_{(-p,q)}^{-1}\circ \tau_{-k/n}\circ\sigma_{(p',q')}(c_1))&=-\textnormal{Re}(c_1)+O(|A|^{-2}).
\end{align*}

By Lemma \ref{matrixfinite}, the number of possible choices for $\sigma_{(-p,q)}^{-1}\circ \tau_{-k/n}\circ\sigma_{(p',q')}$ is finite, thus we assume $\begin{pmatrix}\alpha & \beta \\ \gamma &\delta\end{pmatrix}$ is fixed. From
\begin{align*}
\begin{pmatrix}-s & r \\q &-p\end{pmatrix}\begin{pmatrix}\alpha & \beta \\ \gamma &\delta\end{pmatrix}=\begin{pmatrix}1 & -k/n \\0 &1\end{pmatrix}\begin{pmatrix}s' & r' \\q' &p'\end{pmatrix},
\end{align*}
it follows that $(p',q')=(-\delta p+\beta q,-\gamma p+\alpha q).$ As given in the proof of Theorem \ref{main2}, if we let $(\tilde{u}, \tilde{v}):=(\delta u+\gamma v,\beta u+\alpha v)$ where $\tilde{v}=\tilde{c_1}\tilde{u}+\tilde{c_3}\tilde{u}^3+\cdots$, then
$$\tilde{c_1}=\frac{\beta+\alpha c_1}{\delta+\gamma c_1}=-\overline{c_1},$$
and the $(p',q')$-Dehn filling with respect to $\mathbf{m}, \mathbf{l}$ corresponds to the $(-p,q)$-Dehn filling with respect to the new parametrization. Therefore, 
\begin{align*}
\Theta(p',q')=\tilde{\Theta}(-p,q)=-\text{Im}(\tilde{c_1})\frac{\pi^2}{|\tilde{A}|^2}+\frac{1}{4}\text{Im}\Biggl[\frac{1}{2}\tilde{c_3}\left(\frac{2\pi}{\tilde{A}}\right)^4+\cdots\Biggr]
\end{align*}
where $\tilde{A}=-p+\tilde{c_1}q=-\overline{A}$. Comparing this with
\begin{align*}
\Theta(p,q)=-\text{Im}(-\overline{c_1})\frac{\pi^2}{|A|^2}-\frac{1}{4}\text{Im}\Biggl[\overline{\frac{1}{2}c_3\left(\frac{2\pi}{A}\right)^4}+\cdots\Biggr],
\end{align*}
it follows that $\tilde{c_j}=-\overline{c_j}$ for $j\geq3$. As a result, 
\begin{align*}
f(u_0')&=\textnormal{vol}_{\mathbb{C}}\,M{(p',q')}-\text{vol}_{\mathbb{C}}\,M+\frac{\pi}{2}\ell(\boldsymbol{\gamma}(p',q'))&&\mod i\pi^2\mathbb{Z}\\
&=\frac{i}{4}\left[
\frac{1}{2}\tilde{c_3}\left(\frac{2\pi}{\tilde{A}}\right)^4+\frac{2}{3}\left(3\tilde{c_3}^2\frac{q}{\tilde{A}}-\tilde{c_5}\right)\left(\frac{2\pi}{\tilde{A}}\right)^6+\cdots\right]\\
=\overline{f(u_0)}&=\overline{\textnormal{vol}_{\mathbb{C}}\,M{(p,q)}-\text{vol}_{\mathbb{C}}\,M+\frac{\pi}{2}\ell(\boldsymbol{\gamma}(p,q))}&&\mod i\pi^2\mathbb{Z}, 
\end{align*}
which concludes $\ell(\boldsymbol{\gamma}(p,q))=\overline{\ell(\boldsymbol{\gamma}(p',q'))}$ modulo $\frac{2\pi i}{n}\mathbb{Z}$.

The proof of the converse follows the same argument, as in the proof of Theorem \ref{main2}.
\end{proof}

As a corollary, we attain the following, which implies our second main result, Theorem \ref{25041802}:

\begin{cor}\label{conjucor}
Let $M$ be a 1-cusped hyperbolic 3-manifold and $\mathcal{X}_{ana}$ be its analytic holonomy set. Let $M(p',q')$ and $M(p,q)$ be two Dehn fillings of $M$ with sufficiently large $|p'|+|q'|$ and $|p|+|q|$. Then the following three statements are equivalent: 
\begin{enumerate}
\item $\textnormal{vol}_{\mathbb{C}}\,M-\textnormal{vol}_{\mathbb{C}}\,M(p,q)=\overline{\textnormal{vol}_{\mathbb{C}}\,M-\textnormal{vol}_{\mathbb{C}}\,M{(p',q')}}$ modulo $\frac{i\pi^2}{n}\mathbb{Z}$; 
\item $\ell(\boldsymbol{\gamma}(p,q))=\overline{\ell(\boldsymbol{\gamma}(p',q'))}$ modulo $\frac{2\pi i}{n}\mathbb{Z}$; 
\item there exists $\sigma:\mathcal{X}_{ana}\xrightarrow{\cong} \overline{\mathcal{X}}_{ana}$ such that the Dehn filling point of $M(p',q')$ on $\mathcal{X}_{ana}$ corresponds under $\sigma$ to the Dehn filling point of $M(p,q)$ on $\overline{\mathcal{X}}_{ana}$. 
\end{enumerate}
\end{cor}

The proof of the corollary follows by mimicking that of Corollary \ref{maincor}, so we skip it here. 

\section{Experiments}\label{exp}

We refer to the census manifolds using the nomenclature adopted by \texttt{SnapPy}. The first letter denotes the number of tetrahedra: m, s, v, t, and o9 for $\leq$5, 6, 7, 8, and 9, respectively. The number following the letter indicates the order of volume increase.

We ran a \texttt{Sage} program for the experiments, and the code is publicly available on \texttt{GitHub} \cite{JS}. Our implementation was partially inspired by the program of Futer, Purcell, and Schleimer \cite{FPS}, which is designed to detect cosmetic surgeries on knots and 1-cusped census manifolds.

For a given manifold $M$, we compute $\textnormal{vol}\,M(p,q)$ for $0\leq|p|,|q|\leq100$. We then compare all pairs of Dehn fillings to determine whether they have the same volume. Here, “same” means that the difference between two values is less than or equal to $10^{-62}$, which is very close to \texttt{SnapPy}’s high-precision setting.

If there are fewer than five pairs of Dehn fillings with the same volume, the program outputs them as the exceptional matchings, which are also available along with the code. (The process takes approximately 2 to 6 seconds of wall time per manifold.) Otherwise, we assume the manifold possesses volume symmetry and find the explicit formulas by hand.

As shown in Table \ref{disNM}, there are 214 manifolds with $\tilde{N}^{100}_M>1$, all of which exhibit symmetries in their volume formulae. The results are presented in Table \ref{TableExp}. The matrix $\begin{pmatrix} \alpha & \beta\\ \gamma & \delta \end{pmatrix}$ in the first column implies 
\begin{equation}\label{25042803}
\textnormal{vol}\,M(p,q)=\textnormal{vol}\,M(\alpha p+\beta q,\gamma p+\delta q).
\end{equation}
In the third column, the cusp shapes of the manifolds are listed. If an equation is present, it indicates that the manifolds in that row may have different cusp shapes but all satisfy the given equation. In Proposition \ref{25033001} below, we elaborate more on this equation in light of Conjecture \ref{25030605}.

\begin{longtable}{l|l|F{0.12\textwidth}|F{0.38\textwidth}}
Symmetry of vol$\,M$ & $\tilde{N}^{100}_M$ & $c_1$ & Manifolds \\ \hhline{====}
$\left\langle \begin{pmatrix} 0 & -1\\ 1 & 1 \end{pmatrix}, \begin{pmatrix} 0 & 1\\ 1 & 0 \end{pmatrix} \right\rangle$ & 6 & $\omega$ & m208 \\ \hline
$\left\langle \frac{1}{2}\begin{pmatrix} -1 & -3\\ -1 & 1 \end{pmatrix}, \begin{pmatrix} 1 & 0\\ 0 & -1 \end{pmatrix} \right\rangle$ & 4 & $\sqrt{-3}$ & s595, o9\_37014, o9\_37015, o9\_40998, o9\_41002 \\ \hline
$\left\langle \begin{pmatrix} 0 & -1\\ 1 & 0 \end{pmatrix}, \begin{pmatrix} 1 & 0\\ 0 & -1 \end{pmatrix} \right\rangle$ & 4 & $i$ & m135, v1859, t07829, t12033, t12035, t12036, t12041, t12043, t12045, t12050, o9\_21441, o9\_22828, o9\_31519, o9\_31521 \\ \hline
$\left\langle \frac{1}{2}\begin{pmatrix} 0 & 4\\ 1 & 0 \end{pmatrix}, \begin{pmatrix} 1 & 0\\ 0 & -1 \end{pmatrix} \right\rangle$ & 4 & $2i$ & m136, m140, v2274, v3066, t07828, t08606, t10079, t10238, t12034, o9\_21999, o9\_23923, o9\_25468, o9\_26139, o9\_26650, o9\_27047, o9\_28317, o9\_28436 \\ \hline
$\left\langle \frac{1}{4}\begin{pmatrix} -2 & 3\\ 4 & 2 \end{pmatrix}, \begin{pmatrix} 1 & 1\\ 0 & -1 \end{pmatrix} \right\rangle$ & 4 & $\frac{1}{2}+i$ & t12051, o9\_31522 \\ \hline
$\left\langle \frac{1}{6}\begin{pmatrix} 0 & 9\\ 4 & 0 \end{pmatrix}, \begin{pmatrix} 1 & 0\\ 0 & -1 \end{pmatrix} \right\rangle$ & 4 & $\frac{3}{2}i$ & t12058 \\ \hline
$\left\langle \frac{1}{3}\begin{pmatrix} 3 & 2\\ 0 & -3 \end{pmatrix}, \frac{1}{4}\begin{pmatrix} -1 & 5\\ 3 & 1 \end{pmatrix} \right\rangle$ & 4 & $\frac{1+4i}{3}$ & t12059 \\ \hline
$\left\langle \frac{1}{12}\begin{pmatrix} -2 & -37\\ 4 & 2 \end{pmatrix}, \begin{pmatrix} 1 & 1\\ 0 & -1 \end{pmatrix} \right\rangle$ & 4 & $\frac{1}{2}+3i$ & t12062 \\ \hline
$\left\langle \frac{1}{3}\begin{pmatrix} 0 & 9\\ 1 & 0 \end{pmatrix}, \begin{pmatrix} 1 & 0\\ 0 & -1 \end{pmatrix} \right\rangle$ & 4 & $3i$ & t12063 \\ \hline
$\begin{pmatrix} 0 & -1\\ 1 & 0 \end{pmatrix}$ & 2 & $i$ & m130, m139, v3318, t12038, t12040, o9\_17193, o9\_19556, o9\_35959, o9\_41335, o9\_42724 \\ \hline

$\begin{pmatrix} 1 & 0\\ 0 & -1 \end{pmatrix}$ & 2 & Re$(c_1)$ $=0$ & m004, m009, m206, s118, s119, s464, s726, s882, s892, s912, s915, s956, s961, v2037, v2347, v2787, v2789, v2875, v3103, t03910, t05928, t06828, t07734, t08108, t09253, t10425, t10805, t12071, t12545, t12548, t12587, t12638, t12640, t12648, t12766, t12771, t12819, t12839, o9\_17670, o9\_34677, o9\_35326, o9\_35663, o9\_37643, o9\_39892, o9\_39893, o9\_40736, o9\_41001, o9\_41009, o9\_41933, o9\_43800, o9\_44011, o9\_43569, o9\_44157 \\ \hline

$\begin{pmatrix} 1 & 1\\ 0 & -1 \end{pmatrix}$ & 2 & Re$(c_1)$ $=\frac{1}{2}$ & m003, s955, s957, s960, s772, s775, s778, s779, s787, s911, t12061, t12068, t12069, t12072, t12546, t12547, t12649, t12838, o9\_31514, o9\_31515, o9\_31516, o9\_31517, o9\_36482, o9\_42689, o9\_44013
 \\ \hline

$\begin{pmatrix} 0 & 1\\ 1 & 0 \end{pmatrix}$ & 2 & $|c_1|^2$ $=1$ & m207, s463, s594, s725, s881, s891, v3102, t03911, t05927, t07733, t08107, t09252, t10426, t12639, t12641, o9\_17508, o9\_35327, o9\_40737, o9\_41000, o9\_41004, o9\_41006, o9\_41007, o9\_41008, o9\_43799 \\ \hline

$\frac{1}{2}\begin{pmatrix} 2 & 1\\ 0 & -2 \end{pmatrix}$ & 2 & Re$(c_1)$ $=\frac{1}{4}$ & m137, t10803, o9\_43803, o9\_44093 \\ \hline

$\frac{1}{2}\begin{pmatrix} 2 & -1\\ 0 & -2 \end{pmatrix}$ & 2 & Re$(c_1)$ $=-\frac{1}{4}$ & m410, s784 \\ \hline

$\frac{1}{2}\begin{pmatrix} -2 & 0\\ -1 & 2 \end{pmatrix}$ & 2 & $\scalemath{0.37}{\frac{5-\sqrt{41}+\sqrt{146-26\sqrt{41}}}{4}}$ & o9\_40712, o9\_40715, o9\_40731 \\ \hline

$\frac{1}{2}\begin{pmatrix} -2 & 0\\ 1 & 2 \end{pmatrix}$ & 2 & 4Re$(c_1)$ $=|c_1|^2$ & m010, t12833, o9\_40714, o9\_40730, o9\_40732 \\ \hline

$\frac{1}{2}\begin{pmatrix} -1 & -3\\ -1 & 1 \end{pmatrix}$ & 2 & 2Re$(c_1)+$ $|c_1|^2=3$ & s789, v1540, t04743, t08751, o9\_36481, o9\_43797, o9\_43798, o9\_43808 \\ \hline

$\frac{1}{2}\begin{pmatrix} -1 & 3\\ 1 & 1 \end{pmatrix}$ & 2 & $-$2Re$(c_1)$ $+|c_1|^2$ $=3$ & s788, s916, v1539, v2345, v2346, v3469, v3470, t08752, t09443, t09444, o9\_20566 \\ \hline

$\frac{1}{4}\begin{pmatrix} -3 & -7\\ -1 & 3 \end{pmatrix}$ & 2 & $\sqrt{-7}$ & s777 \\ \hline

$\frac{1}{4}\begin{pmatrix} -3 & 7\\ 1 & 3 \end{pmatrix}$ & 2 & $\sqrt{-7}$ & o9\_39902 \\ \hline

$\frac{1}{4}\begin{pmatrix} -1 & 5\\ 3 & 1 \end{pmatrix}$ & 2 & $\frac{1+\sqrt{-7}}{2}$ & s773, s786 \\ \hline

$\frac{1}{2}\begin{pmatrix} 0 & 4\\ 1 & 0 \end{pmatrix}$ & 2 & $|c_1|^2=4$ & m220, m221, t12835, o9\_17671, o9\_20722, o9\_36681, o9\_36682, o9\_36683, o9\_36684, o9\_39895, o9\_39896 \\ \hline

$\frac{1}{2}\begin{pmatrix} 0 & 4\\ -1 & 0 \end{pmatrix}$ & 2 & $2i$ & v1858, o9\_19555, o9\_31518, o9\_31520 \\ \hline

$\frac{1}{4}\begin{pmatrix} 2 & -4\\ -3 & -2 \end{pmatrix}$ & 2 & $\frac{1+\sqrt{-7}}{2}$ & s781, o9\_39899, o9\_39901 \\ \hline

$\frac{1}{3}\begin{pmatrix} 3 & 0\\ -2 & -3 \end{pmatrix}$ & 2 & $\frac{3+3\sqrt{-7}}{8}$ & o9\_39894 \\ \hline

$\frac{1}{8}\begin{pmatrix} -1 & 13\\ -5 & 1 \end{pmatrix}$ & 2 & $\frac{-1+8i}{5}$ & t12031 \\ \hline

$\frac{1}{6}\begin{pmatrix} 4 & -4\\ -5 & -4 \end{pmatrix}$ & 2 & $\frac{1+3\sqrt{-7}}{8}$ & o9\_39897 \\ \hline

$\frac{1}{6}\begin{pmatrix} -4 & -5\\ -4 & 4 \end{pmatrix}$ & 2 & $\frac{1+3\sqrt{-7}}{8}$ & o9\_39898 \\ \hline
\caption{Classification of census manifolds with $\tilde{N}^{100}_M>1$.} \label{TableExp}
\end{longtable}

Note that the symmetry group in the second row has order 6. Although we have
\begin{align*}
\textnormal{vol}\,M(p,q)&=\textnormal{vol}\,M\left(\frac{-p-3q}{2},\frac{-p+q}{2}\right)=\textnormal{vol}\,M\left(\frac{-p+3q}{2},\frac{p+q}{2}\right)\\
&=\textnormal{vol}\,M(-p,q)=\textnormal{vol}\,M\left(\frac{p+3q}{2},\frac{-p+q}{2}\right)=\textnormal{vol}\,M\left(\frac{p-3q}{2},\frac{p+q}{2}\right)
\end{align*}
for those five manifolds, it is impossible that those six pairs become coprime integer pairs simultaneously. Therefore, we classify them as $\tilde{N}^{100}_M=4$.

As verified in Theorem \ref{25031903}, if $N^{\mathbb{C}}_M(v)>1$ for some $v\in \mathbb{C}$ sufficiently close to $\textnormal{vol}_{\mathbb{C}}\,M$ (modulo $i\pi^2\mathbb{Z}$), then the cusp shape of $M$ must be quadratic; more precisely, contained in either $\mathbb{Q}(\sqrt{-3})$ or $\mathbb{Q}(\sqrt{-1})$. However, in the real volume case, the same property does not hold. In fact, Table \ref{TableExp} shows that there exist $v\in \mathbb{R}$ (sufficiently close to $\text{vol}\,M$) and $M$ with non-quadratic cusp shape such that $N_M(v)>1$. Moreover, it appears that such manifolds $M$ can always be found with cusp shapes of arbitrarily large algebraic degree. 

In the proposition below, we nonetheless show that the cusp shape of $M$ in this case satisfies a strong algebraic relation, which arises from an isomorphism between $\mathcal{X}_{ana}$ and $\overline{\mathcal{X}}_{ana}$.

\begin{prop}\label{25033001}
Let $M$ be a $1$-cusped hyperbolic $3$-manifold whose cusp shape $c_1$ is neither contained in $\mathbb{Q}(\sqrt{-3})$ nor $\mathbb{Q}(\sqrt{-1})$. Assuming Conjecture \ref{25030605}, if there exist $(p,q)$ and $\sigma:=\begin{pmatrix}
\alpha & \beta\\ \gamma & \delta \end{pmatrix}\in \textnormal{GL}_2(\mathbb{Q})$ satisfying \eqref{25042803}, then  
\begin{equation}\label{25031905}
\sigma^2=Id\quad \text{and}\quad \alpha\overline{c}_1+\gamma |c_1|^2=\beta+\delta c_1.
\end{equation}
\end{prop}
\begin{proof}
By the assumption on the cusp shape and applying Theorem \ref{25041802}, $(\sigma^T)^{-1}$ induces an isomorphism from $\mathcal{X}_{ana}$ to $\overline{\mathcal{X}}_{ana}$. Note that, due to symmetry, $(\sigma^T)^{-1}$ also induces an isomorphism from $\overline{\mathcal{X}}_{ana}$ to $\mathcal{X}_{ana}$. Therefore $(\sigma^T)^{-2}\in \textnormal{Aut}\,\mathcal{X}_{ana}$ and hence it must be the identity by Theorem \ref{25031903}.  

To be more explicit, let $\mathcal{X}_{ana}$ and $\overline{\mathcal{X}}_{ana}$ be defined as 
\begin{equation*}
v=c_1u+c_ku^k+\cdots\quad \text{and}\quad \overline{v}=\overline{c}_1\overline{u}+\overline{c}_k\overline{u}^k+\cdots 
\end{equation*}
respectively where $c_k\neq 0$. By the assumption that 
\begin{equation*}
\begin{aligned}
\mathcal{X}_{ana} &\longrightarrow \overline{\mathcal{X}}_{ana}\\
\begin{pmatrix} u \\ v \end{pmatrix}&\mapsto \sigma^T\begin{pmatrix} u \\ v \end{pmatrix}
\end{aligned}
\end{equation*}
is an isomorphism, substituting $\overline{u} = \alpha u + \gamma v$ into the expression for $\overline{v}$ yields the following equality: 
\begin{equation}\label{25050301}
\begin{aligned} 
\overline{v}= \beta u + \delta v \Longrightarrow  \overline{c}_1(\alpha u+\gamma v)+\overline{c}_k(\alpha u+\gamma v)^k+\cdots = \beta u + \delta(c_1 u+c_k u^k+\cdots). 
\end{aligned}
\end{equation} 
Comparing the coefficients of $u$ on both sides, the last equality in \eqref{25031905} are obtained.
\end{proof}

For instance, if $M=\textnormal{o9\_39898}$, the last example in Table \ref{TableExp}, then $c_1=\frac{1+3\sqrt{-7}}{8}, \sigma=\frac{1}{6}\begin{pmatrix} -4 & -5\\ -4 & 4\end{pmatrix}$, and so
$$ \alpha\overline{c}_1+\gamma |c_1|^2=\beta+\delta c_1\Longrightarrow -\frac{4}{6}\frac{1-3\sqrt{-7}}{8}-\frac{4}{6}=-\frac{5}{6}+\frac{4}{6}\frac{1-3\sqrt{-7}}{8}.$$

Indeed, the cusp shape of every example given in Table \ref{TableExp} satisfies this relation, and the equations in the third column, obtained by combining the data from the first column, precisely represent it. 

Note that, comparing the coefficients of $u^k$ on both sides of \eqref{25050301}, we have
$$\overline{c}_k(\alpha+\gamma c_1)^k=(\delta-\gamma\overline{c}_1)c_k.$$
The equality, of course, is also expected to hold for every example in Table \ref{TableExp}.

\section{O-minimality and Transcendentality}\label{SecTr}

In this section and next, we prove our last main result, Theorem \ref{slow}. 

The following is the \textit{Pila–Wilkie counting theorem}, which has been mentioned a couple of times beforehand:

\begin{thm}[\cite{BD}, Theorem 1.1]\label{PWcount}
Let $\mathcal{X}\subseteq \mathbb{R}^n$ with $n\geq1$ be definable in an o-minimal structure on $\mathbb{R}$. Then for all $\varepsilon>0$, there is a constant $C$ such that for all $T$, $$N(\mathcal{X}^{tr},T)\leq CT^\varepsilon.$$
\end{thm}

The \textit{height} of $\frac{a}{b}\in \mathbb{Q}$ is given by $H\left(\frac{a}{b}\right):=\max\{|a|,|b|\}$ where $a$ and $b$ are coprime integers; similarly, for $a=(a_1,\cdots, a_n)\in\mathbb{Q}^n$, $H(a):=\max\{H(a_1), \cdots, H(a_n)\}$. We set $$\mathcal{X}(\mathbb{Q},T):=\{a\in X\cap\mathbb{Q}^n : H(a)\leq T\}$$ and $N(\mathcal{X},T):=\#\mathcal{X}(\mathbb{Q},T)$.

The \textit{algebraic part} $\mathcal{X}^{alg}$ of $\mathcal{X}$ is the union of its connected infinite semialgebraic subsets, that is, sets described by polynomial equations and inequalities (a precise definition will be given later). And the \textit{transcendental part} $\mathcal{X}^{tr}$ of $\mathcal{X}$ is $\mathcal{X}\setminus \mathcal{X}^{alg}$. Heuristically, the above theorem says that rational points are predominantly concentrated in the algebraic part of $\mathcal{X}$, with only a sparse distribution in $\mathcal{X}^{tr}$.

\subsection{O-minimality}

In this subsection, we introduce the theory of o-minimality, where the letter ``o'' stands for “order.” An \textit{ordered field} is a field equipped with a total order that is compatible with its field operations. A standard example is the field of real numbers with the usual order, $(\mathbb{R},\geq)$. Roughly speaking, an o-minimal structure is a well-behaved, axiomatized (model-theoretic) structure over an ordered field that is, in some sense, ``minimal" with respect to its order. We omit the full model-theoretic definition and instead describe it as a collection of sets satisfying certain axioms. The formulation below follows Appendix of \cite{BD} and Section 2 of \cite{J2}.

A \textit{semialgebraic} subset of $\mathbb{R}^n$ is defined to be a finite union of sets of the form $$\{a\in\mathbb{R}^n : g(a)=0, h_1(a)>0, \cdots, h_m(a)>0\}$$ where $g, h_1, \cdots, h_m\in\mathbb{R}[x_1, \cdots, x_n]$.

\begin{dfn}
An \textit{o-minimal structure} on $\mathbb{R}$ is a sequence $\mathcal{S}=(\mathcal{S}_n)$ such that for all $n\geq1$,
\begin{enumerate}
    \item $\mathcal{S}_n$ is a boolean algebra of subsets of $\mathbb{R}^n$.
    \item $\mathcal{S}_n$ contains every semialgebraic subsets of $\mathbb{R}^n$, and $\mathcal{S}_1$ contains no other subsets of $\mathbb{R}$.
    \item If $\mathcal{X}\in\mathcal{S}_n$ and $\mathcal{Y}\in\mathcal{S}_m$, then $\mathcal{X}\times \mathcal{Y}\in\mathcal{S}_{n+m}$.
    \item If $\mathcal{X}\in\mathcal{S}_{n+1}$, then $\pi(\mathcal{X})\in\mathcal{S}_{n}$, where $\pi:\mathbb{R}^{n+1}\rightarrow\mathbb{R}^n$ is the projection onto the first $n$ coordinates.
\end{enumerate}
If $\mathcal{X}\in\mathcal{S}_n$ for some $n$, then we say that $\mathcal{X}$ is \textit{definable} in $\mathcal{S}$.
\end{dfn}

For instance, $\mathcal{S}_1$, the semialgebraic subsets of $\mathbb{R}$, is the collection of finite unions of points and intervals.

One of the nice properties of definable sets that we will employ in the proofs is the following:

\begin{prop}[\cite{BD}, Proposition A.4]\label{finiteness}
Let $\mathcal{Z}\subset\mathbb{R}^m\times\mathbb{R}^n$ be definable in an o-minimal structure, and for each $a\in\mathbb{R}^m$, put
\begin{align*}
\mathcal{Z}_a:=\{b\in\mathbb{R}^n : (a,b)\in \mathcal{Z}\}.
\end{align*}
Then there exist a number $P=P(\mathcal{Z})\in\mathbb{N}$ such that for each $a\in\mathbb{R}^m$ the fiber $\mathcal{Z}_a$ has at most $M$ connected components.
\end{prop}

In the proof of Theorem \ref{PWcount}, the mean value theorem is repeatedly used to bound the determinant of a certain matrix. For a detailed explanation, see the description below Theorem \ref{hype} in Section \ref{sub:key_idea}. A crucial step in this method involves parametrizing the given set with control over the derivatives, ensuring they remain uniformly bounded. This motivates the definition of a  \textit{strong $k$-parametrization} of a set $\mathcal{X}\subseteq\mathbb{R}^n$: a $C^k$-map $f:(0,1)^m\rightarrow\mathbb{R}^n$ such that $f((0,1)^m)=\mathcal{X}$ and $|f^{(\alpha)}|\leq1$ for all $\alpha=(\alpha_1, \cdots, \alpha_n)\in\mathbb{N}^m$ with $\alpha_1+ \cdots+\alpha_n\leq k$.

Another useful property of definable sets that we will adopt is 

\begin{thm}[\cite{BD}, Theorem 1.3]\label{param}
Let $\mathcal{Z}\subset\mathbb{R}^m\times\mathbb{R}^n$ with $n\geq1$ be definable in an o-minimal structure. If $\mathcal{Z}_a$ is a subset of $[-1,1]^n$ with empty interior for all $a\in\mathbb{R}^m$, then there exists a number $P=P(\mathcal{Z})\in\mathbb{N}$ such that for each $a\in\mathbb{R}^m$ the fiber $\mathcal{Z}_a$ is the union of at most $P$ subsets, each having a strong k-parametrization.
\end{thm}

We use the o-minimal structure $\mathbb{R}_{an}$, the smallest structure that contains every $f:[-1,1]^n\rightarrow\mathbb{R}$ that extends to a real analytic function $U\rightarrow\mathbb{R}$ on some open neighborhood $U\subseteq\mathbb{R}^n$ of $[-1,1]^n$, for $n\geq1$.

The following lemma guarantees that $\mathbb{R}_{an}$ suffices for our purpose.  
\begin{lem}[\cite{J2}, Lemma 2.13]\label{ThetaDefinable} Let $M$ be a 1-cusped hyperbolic 3-manifold and $\Theta_M$ be as in Section \ref{NZvf}. Then $\Theta_M(x,y)$ is definable in $\mathbb{R}_{an}$.
\end{lem}

\subsection{Transcendentality}

In order to apply Theorem \ref{PWcount} to a set $\mathcal{X}$, it is necessary to determine its transcendental part $\mathcal{X}^{tr}$. In this subsection, we prove two propositions concerning the transcendence of the Neumann–Zagier volume formula and of a related set.

\begin{prop}\label{transc1}
Let $M$ be a 1-cusped hyperbolic 3-manifold and $\Theta(x,y)$ be as in Section \ref{NZvf}. There is no nonzero $f(x,y,z)\in\mathbb{C}[x,y, z]$ such that
\begin{align*}
    f(x,y,\Theta(x,y))=0
\end{align*}
for sufficiently large $|x|+|y|$.
\end{prop}

\begin{prop}\label{transc2}
Let $M$ be a 1-cusped hyperbolic 3-manifold with its cusp shape $c_1\notin\mathbb{Q}(i)$ and $\Theta(x,y)$ be as in Section \ref{NZvf}. There is no $f(x,y,z,w)\in\mathbb{C}[x,y,z,w]$ such that 
$$\Theta(x,y)=\Theta(z,w)\Longleftrightarrow f(x,y,z,w)=0$$
for sufficiently large $|x|+|y|$ and $|z|+|w|$.
\end{prop}

To prove the propositions, we make use of Ax's theorem, a power series analogue of Schanuel's conjecture and one of the most powerful tools for showing the transcendence of analytic sets.

We say that elements $y_1, \cdots, y_n$ of a $\mathbb{Q}$-vector space have a \textit{$\mathbb{Q}$-linear relation} if $a_1y_1+\cdots+a_ny_n=0$ for some $(a_1,\cdots,a_n)\in\mathbb{Q}^n\setminus\{(0,\cdots,0)\}$; otherwise, they are \textit{$\mathbb{Q}$-linearly independent}. Let $\mathbb{F}$ be an extension field of $\mathbb{E}$. A subset $S\subset\mathbb{F}$ is said to be \textit{$\mathbb{E}$-algebraically independent} if its elements satisfy no nontrivial polynomial relation with coefficients in $\mathbb{E}$. The \textit{transcendence degree} of $\mathbb{F}$ over $\mathbb{E}$, denoted $\textnormal{dim}_\mathbb{E}\,\mathbb{F}$, is defined to be the maximal cardinality of an $\mathbb{E}$-algebraically independent subset of $\mathbb{F}$.

\begin{thm}[\cite{A}]\label{axthm}
Let $y_1, \cdots, y_n\in t\mathbb{C}[[t]]$ be $\mathbb{Q}$-linearly independent. Then
\begin{align*}
    \textnormal{dim}_{\mathbb{C}(t)}\,\mathbb{C}(t)(y_1, \cdots, y_n, e^{y_1}, \cdots, e^{y_n})\geq n.
\end{align*}
\end{thm}

Here, $\mathbb{C}[[t]]$ denotes the ring of formal power series with complex coefficients.

Recall that $m=e^u$ and $l=e^v$ parametrize the holonomy variety $\mathcal{X}$, which implies $e^t$ and $e^{c_1t+c_3t^3+\cdots}$ are $\mathbb{Q}$-algebraically independent. By symmetry, $e^t$ and $e^{\overline{c_1}t+\overline{c_3}t^3+\cdots}$ are also $\mathbb{Q}$-algebraically dependent, and hence so are $e^t$ and $e^{\text{Im}(c_1)t+\text{Im}(c_3)t^3+\cdots}$.

The following lemma is analogous to, but stronger than, Lemma 2.3 of \cite{J2}. The proof is based on a similar idea.

\begin{lem}\label{Qlin}
In $\mathbb{C}[[t]]$, $t$ and $\sum \textnormal{Im}(c_n)t^n$ are linearly independent.
\end{lem}

To prove the lemma, we introduce the Gelfond-Schneider theorem in transcendental number theory.

\begin{thm}[\cite{B}, Theorem 2.4 in the $n=1$ case]\label{GelSch}
If $\alpha$ and $\beta$ are elements in $\overline{\mathbb{Q}}$ such that $\alpha\neq0,1$ and  $\beta\notin\mathbb{Q}$, then $\alpha^\beta$ is transcendental.
\end{thm}

\begin{proof}[Proof of Lemma \ref{Qlin}]
First, note that if $t$ and $\sum \text{Im}(c_n)t^n$ are linearly dependent in $\mathbb{C}[[t]]$, then $\text{Im}(c_n)=0$ for all $n\geq3$. 

Since $e^t$ and $e^{c_1t+c_3t^3+\cdots}$ are $\mathbb{Q}$-algebraically dependent, if we set $t=\frac{2\pi i}{p}$ with sufficiently large $p\in\mathbb{N}$, then $e^t, e^{c_1t+c_3t^3+\cdots}\in \overline{\mathbb{Q}}$ and so $|e^{c_1t+c_3t^3+\cdots}|=e^{\textnormal{Re}(c_1t+c_3t^3+\cdots)}\in \overline{\mathbb{Q}}$. 

Since $\textnormal{Re}(t)=0$ and $\text{Im}(c_n)=0$ for $n\geq3$, 
$$\textnormal{Re}(c_1t+c_3t^3+\cdots)=\textnormal{Re}(c_1t)$$
and so 
\begin{align}\label{25043006}
    e^{\textnormal{Re}(c_1t+c_3t^3+\cdots)}=e^{\textnormal{Re}(c_1t)}=e^{-\frac{2\pi}{p}\textnormal{Im}(c_1)}=(e^{\pi i})^{\frac{2i}{p}\textnormal{Im}(c_1)}=(-1)^{\frac{2i}{p}\textnormal{Im}(c_1)}.
\end{align}
But the last number in \eqref{25043006} is transcendental by Theorem \ref{GelSch}, contradicting the previous claim.  
\end{proof}

Now we prove Proposition \ref{transc1}.

\begin{proof}[Proof of Proposition \ref{transc1}]
Suppose there is such $f(x,y,z)\in\mathbb{C}[x,y, z]$. By substituting $y=0$, we get $f(x,0,\Theta(x,0))=0$ where
\begin{align*}
    \Theta(x,0)=\frac{-\text{Im}(c_1)\pi^2}{x^2}+\frac{1}{4}\left(\frac{1}{2}\text{Im}(c_3)\frac{(2\pi)^4}{x^4}-\frac{1}{3}\text{Im}(c_5)\frac{(2\pi)^6}{x^6}+\cdots\right).
\end{align*}
If we let $t=\frac{2\pi i}{x}$, then the following formal power series 
\begin{align*}
    \frac{1}{2}\left(\text{Im}(c_1)t^2+\frac{1}{2}\text{Im}(c_3)t^4+\frac{1}{3}\text{Im}(c_3)t^6+\cdots\right)
\end{align*}
and its formal derivative $\sum \text{Im}(c_n)t^n$ are all algebraic over $\mathbb{C}(t)$. If we set
\begin{align*}
    y_1=t\;\;\text{and}\;\; y_2=\text{Im}(c_1)t+\text{Im}(c_3)t^3+\cdots, 
\end{align*}
then, clearly $y_1, y_2\in t\mathbb{C}[[t]]$, and they are $\mathbb{Q}$-linearly independent by Lemma \ref{Qlin}. Further,  
\begin{align*}
    \textnormal{dim}_{\mathbb{C}(t)}\,\mathbb{C}(t)(y_1, y_2, e^{y_1}, e^{y_2})\geq 2
\end{align*}
by Theorem \ref{axthm}. On the other hand, since both $y_1$ and $y_2$ are algebraic over $\mathbb{C}(t)$, and $e^{y_1}$ and $e^{y_2}$ are $\mathbb{C}$-algebraically dependent, it follows that 
\begin{align*}
    \textnormal{dim}_{\mathbb{C}(t)}\,\mathbb{C}(t)(y_1, y_2, e^{y_1}, e^{y_2})\leq 1,
\end{align*} 
which is a contradiction.
\end{proof}

The proof of Proposition \ref{transc2} is based on a similar argument as the previous one, but requires a few more lemmas. 

We begin with the following lemma, whose proof is due to Seewoo Lee.

\begin{lem}\label{numb1}
Let $c_1$ be a nonreal algebraic number. Then 
\begin{align*}
    |a+bc_1|\notin\mathbb{Q}(\textnormal{Re}(c_1), \textnormal{Im}(c_1))
\end{align*}
for infinitely many coprime integer pairs $(a,b)$.
\end{lem}

To prove the lemma, we apply Hilbert's irreducibility theorem. While the theorem is stated in many different forms, we adapt a simplified version presented in \cite{EMSW}. For more details, please see Proposition 3.3.5 in \cite{S}.

\begin{thm}[\cite{EMSW}, Theorem 2.9 in a simpler form]\label{HilbertIrr}
Let $\mathbb{K}$ be a number field and $f(x_0, x_1, \cdots, x_m)\in \mathbb{K}[x_0, x_1, \cdots, x_m]$ be an irreducible polynomial. There exists a set $\mathcal{T}\subseteq \mathbb{K}^m$ not containing $\mathbb{Q}^m$ such that if $(k_1, \cdots, k_m)\in \mathbb{K}^m\setminus \mathcal{T}$, then $f(x_0, k_1, \cdots, k_m)\in \mathbb{K}[x_0]$ is irreducible.
\end{thm}

The set $\mathcal{T}$ can be taken to have the special property of ``thinness". Intuitively, a \textit{thin} set is one that lies in the image of a map from a lower-dimensional algebraic variety and is, in a sense, ``small'' from an algebraic point of view. The only fact we need, however, is that thin sets in $\mathbb{K}^m$ cannot contain $\mathbb{Q}^m$. See Propositions 3.2.1 and 3.4.1 of \cite{S}.

\begin{proof}[Proof of Lemma \ref{numb1}]
Let $R:=\textnormal{Re}(c_1)$, $I:=\textnormal{Im}(c_1)\neq0$ and $\mathbb{K}:=\mathbb{Q}(R, I)$. Let $\alpha, \beta$ be some rational numbers. Since $|\alpha+\beta c_1|^2=(\alpha+\beta R)^2+(\beta I)^2$, clearly $|\alpha+\beta c_1|\notin \mathbb{K}$ if and only if the polynomial $x_0^2-(\alpha+\beta R)^2-(\beta I)^2$ is irreducible over $\mathbb{K}$.

Now consider $f(x_0,x_1,x_2):=x_0^2-(x_1+Rx_2)^2-(Ix_2)^2$. As $x_0^2-x_1^2-x_2^2$ is irreducible over $\mathbb{C}$, it follows that $f$ is irreducible over $\mathbb{K}$. Applying Theorem \ref{HilbertIrr}, we obtain $(\alpha, \beta)\in\mathbb{Q}^2$ such that $f(x_0,\alpha,\beta)\in \mathbb{K}[x_0]$ is irreducible. Therefore $|\alpha+\beta c_1|\notin \mathbb{K}$, which further implies there exists a coprime pair $(a,b)\in \mathbb{Z}^2$ such that $|a+bc_1|\notin \mathbb{K}$.

Replacing $\mathbb{K}$ by $\mathbb{K}':=\mathbb{K}(|a+bc_1|)$, the same argument yields another coprime pair $(a',b')\in \mathbb{Z}^2$ such that $|a'+b'c_1|\notin \mathbb{K}'$. Repeating the process will complete the proof.
\end{proof}

The following lemma, which is proved using the lemma above, is elementary. It will be applied in the proof of Proposition \ref{transc2} to establish the $\mathbb{Q}$-linear independence of certain series.

\begin{lem}\label{numb2}
Let $c_1$ be an algebraic number such that $c_1\notin\mathbb{Q}(i)$. Then there exist infinitely many coprime integer pairs $(a,b)$ such that if there is a nontrivial integer solution $(\alpha, \beta, \gamma, \delta)$ of the equation
\begin{align}\label{Qlineq}
    \alpha+\beta\textnormal{Im}(c_1)+\gamma|a+bc_1|+\delta\frac{\textnormal{Im}(c_1)}{|a+bc_1|}=0,
\end{align}
then either $\alpha=\beta=0$ or $\gamma=\delta=0$.
\end{lem}

\begin{proof}
Suppose there are only finitely many such pairs $(a,b)$. That is, for all but finitely many coprime pairs $(a,b)\in\mathbb{Z}^2$, there exists a nontrivial integer solution $(\alpha, \beta, \gamma, \delta)$ to \eqref{Qlineq} with $(\alpha,\beta)\neq(0,0)$ and $(\gamma,\delta)\neq(0,0)$. Multiplying $|a+bc_1|$ to \eqref{Qlineq}, we obtain
\begin{align*}
    (\alpha+\beta\textnormal{Im}(c_1))|a+bc_1|+\gamma|a+bc_1|^2+\delta\textnormal{Im}(c_1)=0.
\end{align*}
Note that $|a+bc_1|^2=a^2+2ab\textnormal{Re}(c_1)+b^2|c_1|^2$. If $\alpha+\beta\textnormal{Im}(c_1)\neq0$, then $|a+bc_1|\in\mathbb{Q}(\textnormal{Re}(c_1), \textnormal{Im}(c_1))$. Thus, by Lemma \ref{numb1}, there are infinitely many coprime pairs $(a,b)$ such that $\alpha+\beta\textnormal{Im}(c_1)=0$, which implies $\textnormal{Im}(c_1)\in \mathbb{Q}$ as well as $$2a\textnormal{Re}(c_1)+b|c_1|^2=\frac{1}{b}\left(-\frac{\delta}{\gamma}\textnormal{Im}(c_1)-a^2\right)\in \mathbb{Q}.$$ In conclusion, $\textnormal{Re}(c_1)\in \mathbb{Q}$ and so $c_1\in\mathbb{Q}(i)$. But this contradicts our initial assumption.
\end{proof}

\begin{proof}[Proof of Proposition \ref{transc2}]
Let $a$ and $b$ be integers satisfying the condition of Lemma \ref{numb2}, and $c$ and $d$ be integers satisfying $ad-bc=1$. Let $\tilde{\Theta}$ be the function defined in \eqref{Thetatilde} with respect to the basis change $\tilde{\mathbf{m}}=\mathbf{m}^a\mathbf{l}^b, \tilde{\mathbf{l}}=\mathbf{m}^c\mathbf{l}^d$.

Suppose there is a polynomial satisfying the given condition. Through a change of variables, we may assume that there exists $f(x,y,z,w)\in\mathbb{C}[x,y, z,w]$ such that 
$$\Theta(x,y)=\tilde{\Theta}(z,w)\Longleftrightarrow f(x,y,z,w)=0$$ 
for sufficiently large $|x|+|y|$ and $|z|+|w|$. By substituting $y=w=0$, we get 
$$\Theta(x,0)=\tilde{\Theta}(z,0)\Longleftrightarrow  f(x,0,z,0)=0$$ 
for sufficiently large $|x|$ and $|z|$, where
\begin{align*}
    \Theta(x,0)&=\frac{-\text{Im}(c_1)\pi^2}{x^2}+\frac{1}{4}\left(\frac{1}{2}\text{Im}(c_3)\frac{(2\pi)^4}{x^4}-\frac{1}{3}\text{Im}(c_5)\frac{(2\pi)^6}{x^6}+\cdots\right),\\
    \tilde{\Theta}(z,0)&=\frac{-\text{Im}(\tilde{c_1})\pi^2}{z^2}+\frac{1}{4}\left(\frac{1}{2}\text{Im}(\tilde{c_3})\frac{(2\pi)^4}{z^4}-\frac{1}{3}\text{Im}(\tilde{c_5})\frac{(2\pi)^6}{z^6}+\cdots\right).
\end{align*}

Simplifying the notation by letting $t=\frac{2\pi i}{x}$ and $s=\frac{2\pi i}{z}$, we get the following curve
\begin{align}\label{setst}
\left\{(t,s)\in\mathbb{R}^2 : \sum_{n=1}^\infty\frac{\text{Im}(c_n)}{n+1}t^{n+1}=\sum_{n=1}^\infty\frac{\text{Im}(\tilde{c_n})}{n+1}s^{n+1}\right\}
\end{align}
is also algebraic. Near $t=0$, if one may view $s$ as a formal power series in $t$, then, since $\text{Im}(\tilde{c_n})=\frac{\text{Im}(c_n)}{|a+bc_1|^2}$, it follows that $s=\pm|a+bc_1|t+\cdots.$ By a slight abuse of notation, we choose one branch (the positive sign) of the algebraic curve and denote it by $s$. Then both $s$ and $ds/dt$ are algebraic over $\mathbb{C}(t)$. By differentiating the defining formula of \eqref{setst} by $t$, it follows that 
\begin{align}\label{Ctdep}
\sum_{n=1}^\infty\text{Im}(c_n)t^{n}-\frac{ds}{dt}\sum_{n=1}^\infty\text{Im}(\tilde{c_n})s^{n}=0, 
\end{align}
which implies $\sum \text{Im}(c_n)t^n$ and $\sum \text{Im}(\tilde{c_n})s^n$ are $\mathbb{C}(t)$-algebraically dependent.

Now, let
\begin{align*}
    y_1&=t,&&\\
    y_2&=\text{Im}(c_1)t+\text{Im}(c_3)t^3+\cdots,&&\\
    y_3&=s&&=|a+bc_1|t+\cdots,\\
    y_4&=\text{Im}(\tilde{c_1})s+\text{Im}(\tilde{c_3})s^3+\cdots&&=\frac{\text{Im}(c_1)}{|a+bc_1|}t+\cdots.
\end{align*}
Clearly, $y_1, y_2, y_3, y_4\in t\mathbb{C}[[t]]$. If they are $\mathbb{Q}$-linearly dependent, then the relation of their coefficients of the 1st order terms is of the form given in \eqref{Qlineq}. By Lemma \ref{numb2}, such a $\mathbb{Q}$-linear relation can only occur between $t$ and $\sum \text{Im}(c_n)t^n$, or between $s$ and $\sum \text{Im}(\tilde{c_n})s^n$; but both contradict Lemma \ref{Qlin}. Consequently, $y_1, y_2, y_3$, and $y_4$ are $\mathbb{Q}$-linearly independent, and so 
\begin{align}\label{25042001}
    \textnormal{dim}_{\mathbb{C}(t)}\,\mathbb{C}(t)(y_1, y_2, y_3, y_4, e^{y_1}, e^{y_2}, e^{y_3}, e^{y_4})\geq 4
\end{align}
by Theorem \ref{axthm}. 

On the other hand, both $y_1$ and $y_3$ are algebraic over $\mathbb{C}(t)$, and $y_2$ and $y_4$ are $\mathbb{C}(t)$-algebraically dependent from \eqref{Ctdep}. Moreover, as explained in the discussion preceding Lemma \ref{Qlin}, $e^{y_1}$ and $e^{y_2}$ are $\mathbb{C}$-algebraically dependent, and similarly for $e^{y_3}$ and $e^{y_4}$. Therefore,
\begin{align*}
    \textnormal{dim}_{\mathbb{C}(t)}\,\mathbb{C}(t)(y_1, y_2, y_3, y_4, e^{y_1}, e^{y_2}, e^{y_3}, e^{y_4})\leq 3,
\end{align*}
contradicting \eqref{25042001}.
\end{proof}

\noindent \textbf{Remark. } In fact, what we proved in the preceding proof is somewhat stronger than the statement of the proposition. The proof essentially claims that the following curve
\begin{align}\label{abcurve}
    \{(x,z)\in\mathbb{R}^2 : \textnormal{vol}\,M(x,0)=\textnormal{vol}\,M(az,bz)\}
\end{align}
is transcendental for infinitely many coprime integer pairs $(a,b)$. Recall from Section \ref{NZvf} that $\tilde{\Theta}(x,y)=\Theta(az+cw,bz+dw)$. Therefore, if the curve $\{(x,z)\in\mathbb{R}^2 : \Theta(x,0)=\tilde{\Theta}(z,0)\}$ in the proof is transcendental, then the curve \eqref{abcurve} is also transcendental.

\section{Proof of Theorem \ref{slow}}\label{SecSlow}

Finally, we prove our last main result, Theorem \ref{slow}, in this section. Before embarking on the proof, we introduce a few more ingredients and outline its key idea. 

\subsection{A few more ingredients and the key idea}\label{sub:key_idea}

The following is an intermediate result of the Pila–Wilkie counting theorem. A \textit{hypersurface in $\mathbb{R}^n$ of degree $\leq e$} is defined to be the zero set in $\mathbb{R}^n$ of a nonzero $n$-variable polynomial over $\mathbb{R}$ of total degree $\leq e$.

\begin{thm}[\cite{BD}, Theorem 1.2]\label{hype}
Let $n\geq1$ be given. Then for any $e\geq1$ there are  $k=k(n,e)\geq1$, $\varepsilon:=\varepsilon(n,e)$, and $C=C(n,e)$, such that if $\mathcal{X}\subseteq \mathbb{R}^n$ has a strong $k$-parametrization, then for all $T$ at most $CT^\varepsilon$ many hypersurfaces in $\mathbb{R}^n$ of degree $\leq e$ are enough to cover $\mathcal{X}(\mathbb{Q},T)$, with $\varepsilon(n,e)\rightarrow0$ as $e\rightarrow\infty$.
\end{thm}

Those hypersurfaces are given by the ``determinant method" of Bombieri and Pila \cite{BP}. For the future use, let us briefly review the method. Let $n=2, J:=\{(i,j):0\leq i,j\leq e \text{ and } i+j\leq e\}$ and $D:=|J|=\frac{(e+1)(e+2)}{2}$. 

Take $D$ rational points $(p_1, q_1), \cdots, (p_D,q_D)\in \mathcal{X}(\mathbb{Q},T)$. Bombieri and Pila observed that if these points are sufficiently close to one another, then the determinant of the $D\times D$ matrix
\begin{align}
\left(p_l^iq_l^j\right)_{1\leq l\leq D, (i,j)\in J}\label{BPmatrix}
\end{align}
must vanish, as it is a rational number of bounded height. In that case, a nontrivial solution $(a_{i,j})_{(i,j)\in J}$ to the system of linear equations produced by the matrix \eqref{BPmatrix} yields a polynomial $$\sum_{(i,j)\in J}a_{i,j}x^iy^j=0$$ defining a hypersurface in $\mathbb{R}^2$ of degree $\leq e$ that contains all the chosen points.

From now on, we introduce a more general version of the Pila–Wilkie theorem that counts arbitrarily algebraic points of a given set, extending beyond rational points. As in the original version, the number of such points grows with the height, but in this case, it additionally requires controlling one more variable: the algebraic degree of a point. 

Before stating the theorem, let us go over some relevant definitions. For $d\geq1$, the height of an algebraic number $\alpha\in\mathbb{R}$ with $[\mathbb{Q}(\alpha):\mathbb{Q}]\leq d$ is defined by 
\begin{align*}
H^{poly}_d(\alpha):=\min\left\{H(\xi):\xi\in\mathbb{Q}^d, \alpha^d+\xi_1\alpha^{d-1}+\cdots\xi_d=0\right\}\in\mathbb{N}^{\geq1}.
\end{align*}
For general $\alpha\in\mathbb{R}$, $H^{poly}_d(\alpha):=\infty$ when $[\mathbb{Q}(\alpha):\mathbb{Q}]> d$. As in the rational case, if $\alpha=(\alpha_1,\cdots, \alpha_n)\in\mathbb{R}^n$, then $H^{poly}_d(\alpha):=\max\{H^{poly}_d(\alpha_1), \cdots, H^{poly}_d(\alpha_n)\}$. Lastly, for $\mathcal{Y}\subseteq\mathbb{R}^n$, let
\begin{align*}
\mathcal{Y}_d(T):=\{\alpha\in \mathcal{Y} : H^{poly}_d(\alpha)\leq T\}\quad \text{and}\quad N_d(\mathcal{Y},T):=\#\mathcal{Y}_d(T).
\end{align*}

The following theorem can be seen as a generalization of Theorem \ref{PWcount}, which corresponds to the case $d=1$.

\begin{thm}[\cite{BD}, Theorem 8.9]\label{countingAlg}
Let $\mathcal{Y}\subseteq\mathbb{R}^n$ be definable in an o-minimal structure, and let $\varepsilon>0$ be given. Then there is a constant $C=C(\mathcal{Y},d,\varepsilon)$ such that for all $T$,
\begin{align*}
N_d(\mathcal{Y}^{tr},T)\leq C T^\varepsilon.
\end{align*}
\end{thm}

While Theorem \ref{slow} might initially appear to be a straightforward consequence of the Pila–Wilkie counting theorem, the situation is more delicate. For instance, a direct application of the theorem requires a classification of all semialgebraic subsets of $$\{(x,y,z,w)\in\mathbb{R}^4:\Theta(x,y)=\Theta(z,w)\},$$ a task that is already highly complex and difficult to carry out explicitly.

Instead, we proceed indirectly. By focusing on the level curves of the volume formula, we show, via Theorem \ref{hype}, that these curves are generically algebraic when the growth rate of $N_M$ is sufficiently fast. Furthermore, we prove that their defining polynomials have uniformly bounded degrees and rational coefficients with controlled heights (see Claim \ref{algcurveclm}). Using these results, we then construct many algebraic points on the transcendental curve $\mathcal{Y}$, defined in \eqref{abcurve}, which are sufficiently abundant to deduce a contradiction with the conclusion of Theorem \ref{countingAlg}.

\subsection{Proof}

Now we are ready to prove Theorem \ref{slow}.

\begin{proof}[Proof of Theorem \ref{slow}]
To get a contradiction, suppose not, i.e. for some $\delta>0$, there exists a sequence of sets of distinct Dehn filling coefficients $$\{\{(p_{n,1}, q_{n,1}), \cdots, (p_{n,n},q_{n,n})\}\}_{n=1}^\infty$$
satisfying
\begin{align*}
\textnormal{vol} \,M(p_{n,1},q_{n,1})=\cdots=\textnormal{vol} \,M(p_{n,n},q_{n,n})
\end{align*}
and $n\geq(|p_{n,1}|+|q_{n,1}|)^\delta$ for sufficiently large $n$. For convenience, let $(p_n,q_n)$ denote $(p_{n,1},q_{n,1})$ where $|p_n|+|q_n|\rightarrow\infty$ as $n\rightarrow\infty$.

We claim
\begin{clm}\label{algcurveclm}
For sufficiently large $n$, the level set of the volume formula
\begin{align*}
L(p_n,q_n):=\{(x,y)\in\mathbb{R}^2 : \textnormal{vol} \,M(x,y)=\textnormal{vol} \,M(p_{n},q_{n})\}
\end{align*}
is an algebraic curve. Moreover, there exist constants $e, Q\in\mathbb{N}$, not depending on $n$, such that the defining polynomial $f_n(x,y)$ of the curve has degree $\leq e$ and rational coefficients of heights $\leq(|p_n|+|q_n|)^{Q}$.
\end{clm}

\begin{proof}[Proof of Claim \ref{algcurveclm}]
To apply Theorems \ref{param} and \ref{hype}, we first change the variables, following the method given in the proof of Lemma \ref{ThetaDefinable} in \cite{J2}.

Fix $C_1>1$. We may assume the coefficients $(p_{n,l},q_{n,l})$ lie in the portion $p>0$ and $-C_1<q/p<C_1$ for all $n$ and $l$. In fact, by changing the basis of $H_1(T^2)$, we can cover the whole $L(p_n,q_n)$ by a finite number of such portions. Let $x'=\frac{1}{x}$ and $y'=\frac{y}{x}$. Consider 
\begin{align*}
\tilde{L}(z,w):=\left\{(x',y')\in(0,1)\times\left(-C_1,C_1\right)\subset\mathbb{R}^2 : \Theta\left(\frac{1}{x'}, \frac{y'}{x'}\right)=\Theta(z,w)\right\}.
\end{align*}
Note that $$\left(\frac{1}{p_{n,1}}, \frac{q_{n,1}}{p_{n,1}}\right), \cdots, \left(\frac{1}{p_{n,n}}, \frac{q_{n,n}}{p_{n,n}}\right)$$
are $n$ distinct rational points of $\tilde{L}(p_n,q_n)$ with heights $\leq C_2(|p_n|+|q_n|)$ for some $C_2>0$ by Lemma \ref{Anear}.

Since $\tilde{L}(p_n,q_n)$ is bounded and definable in $\mathbb{R}_{an}$, by Theorem \ref{param}, it can be covered by a finite number of subsets, each having a strong $k$-parametrization. Moreover, as they are fibers of $\tilde{L}$, the number of subsets can be taken independently from $n$.

Take sufficiently large $e$ such that $\varepsilon(2,e)<\delta$. By Theorem \ref{hype}, there exists $C_3>0$ such that for all $T$ at most $C_3T^\varepsilon$ many hypersurfaces in $\mathbb{R}^2$ of degree $\leq e$ are enough to cover $\tilde{L}(p_n,q_n)(\mathbb{Q},T)$. Define $\mathcal{Z}\subset\mathbb{R}^{D+2}\times\mathbb{R}^2$ to be the set
\begin{align*}
\mathcal{Z}:=\biggl\{&(z,w, a_{e,0},a_{e-1,1}, \cdots, a_{0,0};x',y')\in\mathbb{R}^{D+2}\times\mathbb{R}^2 : \\
&\Theta\left(\frac{1}{x'}, \frac{y'}{x'}\right)=\Theta(z,w)\text{ and } a_{e,0}x'^e+a_{e-1,1}x'^{e-1}y'+\cdots+a_{0,0}=0\biggr\}.
\end{align*}
This is the intersection of a set defined by the definable function $\Theta$ with a semialgebraic set, and is therefore definable in $\mathbb{R}_{an}$. By Proposition \ref{finiteness}, there exist $P\in\mathbb{N}$ such that for each $(p,q, a_{e,0},\cdots, a_{0,0})\in\mathbb{R}^{D+2}$, the fiber $\mathcal{Z}_{(p,q, a_{e,0},\cdots, a_{0,0})}$ has at most $P$ connected components. This implies that for any hypersurface in $\mathbb{R}^2$ of degree $\leq e$, its intersection with $\tilde{L}(p_n,q_n)$ has at most $P$ points unless $\tilde{L}(p_n,q_n)$ itself coincides with that hypersurface.

Since $n(\geq(|p_{n}|+|q_{n}|)^\delta)>C_3P(C_2(|p_n|+|q_n|))^{\varepsilon}$ for $n$ sufficiently large, we conclude each $\tilde{L}(p_n,q_n)$ is equal to one of the hypersurfaces in $\mathbb{R}^2$ of degree $\leq e$ given by Theorem \ref{hype}.

By the discussion below Theorem \ref{hype}, there exists a constsnt $Q>0$, not depending on $n$, such that each $\tilde{L}(p_n,q_n)$ is defined by a rational polynomial $\tilde{f}_n(x',y')$ of degree $\leq e$ whose coefficients have heights $\leq(|p_n|+|q_n|)^{Q}$. The conclusion of the claim follows immediately from this. \end{proof}

From the remark following the proof of Proposition \ref{transc1}, the curve 
\begin{align*}
    \mathcal{Y}:=\{(x,z)\in\mathbb{R}^2 : \Theta(x,0)=\Theta(az,bz)\}
\end{align*}
is transcendental for some $(a,b)\in\mathbb{Z}^2$. Fix such a pair. Then $\mathcal{Y}$ is definable in $\mathbb{R}_{an}$. For each $n$, if $(\alpha_n,\beta_n)\in\mathbb{R}^2$ is a solution to
\begin{align*}
    \Theta(\alpha_n,0)=\Theta(p_n,q_n)=\Theta(a\beta_n,b\beta_n), 
\end{align*}
then $(\alpha_n,0)$ and $(a\beta_n,b\beta_n)$ lie on the algebraic curve $L(p_n,q_n)$, and so
$$f_n(\alpha_n,0)=f_n(a\beta_n,b\beta_n)=0,$$
where $f_n$ is the same as given in the proof of Claim \ref{algcurveclm}. By the same claim, we further get $(\alpha_n,\beta_n)$ is an algebraic point on $\mathcal{Y}$ of degree $\leq e$ with height $\leq(|p_n|+|q_n|)^{Q+1}$ for all sufficiently large $n$. Therefore, defining $T_n:=(|p_n|+|q_n|)^{Q+1}$, we have $(\alpha_n,\beta_n)\in\mathcal{Y}_e(T_n)$ for $n$ large enough. Let $n_0$ denote the smallest index for which this inclusion holds for $n>n_0$.

Take $\varepsilon'<\delta/(2(Q+1))$. By theorem \ref{countingAlg}, there is a constant $C_4$ such that $N_e(\mathcal{Y},T)\leq C_4 T^{\varepsilon'}$ for all $T$. However,
\begin{align*}
    \frac{1}{2}(|p_n|+|q_n|)^\delta\leq\frac{1}{2}n< n-n_0< N_e(\mathcal{Y},T_n)\leq C_4 T_n^{\varepsilon'}<C_4(|p_n|+|q_n|)^{\delta/2}
\end{align*}
for $n>2n_0$, which is a contradiction.
\end{proof}

\section{Examples}\label{exm}

\begin{exm}\label{208comp}
Let $M$ be the manifold m208 which is an arithmetic manifold commensurable to m003 (and hence to the figure eight knot complement m004). The tetrahedral manifold otet$08\_{0010}$ is a 2-fold covering of $M$ and a 4-fold covering of m003. (the later covering is combinatorial. See Figure 2 of \cite{FGGTV}.) The manifold $M$ has the cusp shape $c_1=\omega$. We show that 
\begin{equation}\label{25050103}
\begin{aligned}
    \textnormal{vol}_{\mathbb{C}}\,M(p,q)=\textnormal{vol}_{\mathbb{C}}\,M(-q,p+q)=\textnormal{vol}_{\mathbb{C}}\,M(-p-q,p)\quad&\mod i\pi^2\mathbb{Z},\\
    \textnormal{vol}_{\mathbb{C}}\,M-\textnormal{vol}_{\mathbb{C}}\,M(p,q)=\overline{\textnormal{vol}_{\mathbb{C}}\,M-\textnormal{vol}_{\mathbb{C}}\,M(q,p)}\quad&\mod i\pi^2\mathbb{Z}
\end{aligned}
\end{equation}
for sufficiently large $|p|+|q|$. 

\begin{figure}[htb]
\begin{tikzpicture}[line width=1.2, line cap=round, line join=round, scale=.5]
    \begin{scope}[color=yellow!90!red]
    \draw (8.38, 6.11) .. controls (8.40, 6.44) and (8.40, 6.76) .. 
          (8.33, 7.08) .. controls (8.25, 7.43) and (8.01, 7.71) .. 
          (7.75, 7.97) .. controls (7.47, 8.25) and (7.16, 8.50) .. 
          (6.81, 8.69) .. controls (6.53, 8.85) and (6.22, 8.96) .. 
          (5.91, 9.02) .. controls (5.56, 9.08) and (5.22, 9.09) .. 
          (4.87, 9.07) .. controls (4.46, 9.05) and (4.05, 8.97) .. 
          (3.66, 8.85) .. controls (3.27, 8.73) and (2.90, 8.58) .. 
          (2.57, 8.35) .. controls (2.24, 8.13) and (1.97, 7.84) .. 
          (1.74, 7.51) .. controls (1.55, 7.22) and (1.37, 6.91) .. 
          (1.32, 6.57) .. controls (1.29, 6.33) and (1.32, 6.09) .. (1.45, 5.89);
    \draw (1.60, 5.65) .. controls (1.82, 5.32) and (2.13, 5.07) .. 
          (2.50, 4.93) .. controls (2.88, 4.80) and (3.30, 4.85) .. (3.65, 5.05);
    \draw (3.65, 5.05) .. controls (4.00, 5.25) and (4.40, 5.34) .. 
          (4.80, 5.35) .. controls (5.11, 5.36) and (5.45, 5.37) .. 
          (5.71, 5.22) .. controls (5.88, 5.12) and (6.06, 5.03) .. (6.24, 4.95);
    \draw (6.50, 4.84) .. controls (6.92, 4.66) and (7.39, 4.85) .. 
          (7.73, 5.18) .. controls (8.01, 5.45) and (8.35, 5.78) .. (8.38, 6.11);
  \end{scope}
  \begin{scope}[color=green]
    \draw (6.37, 4.90) .. controls (6.48, 5.17) and (6.65, 5.42) .. 
          (6.84, 5.65) .. controls (7.03, 5.88) and (7.28, 6.07) .. 
          (7.58, 6.14) .. controls (7.80, 6.19) and (8.03, 6.23) .. (8.24, 6.16);
    \draw (8.51, 6.07) .. controls (8.67, 6.02) and (8.90, 5.82) .. 
          (9.08, 5.68) .. controls (9.37, 5.44) and (9.56, 5.10) .. 
          (9.67, 4.75) .. controls (9.82, 4.27) and (9.85, 3.77) .. 
          (9.85, 3.28) .. controls (9.84, 2.73) and (9.70, 2.20) .. 
          (9.40, 1.74) .. controls (9.13, 1.35) and (8.82, 0.99) .. 
          (8.43, 0.72) .. controls (8.07, 0.46) and (7.65, 0.30) .. 
          (7.22, 0.23) .. controls (6.79, 0.16) and (6.36, 0.12) .. 
          (5.93, 0.16) .. controls (5.54, 0.19) and (5.17, 0.32) .. (4.86, 0.55);
    \draw (4.86, 0.55) .. controls (4.51, 0.80) and (4.12, 1.08) .. 
          (4.14, 1.50) .. controls (4.15, 1.74) and (4.20, 1.99) .. 
          (4.34, 2.18) .. controls (4.49, 2.37) and (4.67, 2.54) .. (4.86, 2.69);
    \draw (5.08, 2.87) .. controls (5.26, 3.01) and (5.48, 3.19) .. 
          (5.58, 3.32) .. controls (5.68, 3.46) and (5.78, 3.64) .. 
          (5.87, 3.80) .. controls (6.06, 4.14) and (6.22, 4.52) .. (6.37, 4.90);
  \end{scope}
  \begin{scope}[color=purple!60!blue]
    \draw (3.60, 5.18) .. controls (3.54, 5.35) and (3.45, 5.52) .. 
          (3.31, 5.64) .. controls (3.13, 5.78) and (2.86, 5.85) .. 
          (2.61, 5.91) .. controls (2.25, 6.01) and (1.87, 5.93) .. (1.53, 5.77);
    \draw (1.53, 5.77) .. controls (1.06, 5.54) and (0.63, 5.22) .. 
          (0.39, 4.75) .. controls (0.18, 4.33) and (0.09, 3.85) .. 
          (0.10, 3.37) .. controls (0.12, 2.91) and (0.20, 2.44) .. 
          (0.40, 2.02) .. controls (0.59, 1.63) and (0.85, 1.29) .. 
          (1.17, 1.00) .. controls (1.44, 0.76) and (1.73, 0.53) .. 
          (2.07, 0.39) .. controls (2.43, 0.23) and (2.82, 0.18) .. 
          (3.21, 0.14) .. controls (3.73, 0.10) and (4.26, 0.24) .. (4.73, 0.49);
    \draw (4.98, 0.61) .. controls (5.33, 0.79) and (5.58, 1.13) .. 
          (5.58, 1.52) .. controls (5.59, 2.00) and (5.30, 2.43) .. (4.97, 2.78);
    \draw (4.97, 2.78) .. controls (4.67, 3.11) and (4.36, 3.44) .. 
          (4.15, 3.83) .. controls (3.96, 4.17) and (3.82, 4.54) .. (3.69, 4.91);
  \end{scope}
  \node[left] at (0.2,3) {$(1,3)$};
  \node[right] at (9.8,3) {$(1,3)$};
\end{tikzpicture}
\caption{The manifold m208 which is obtained from the exterior of the link $6_1^3$ by filling two cusps}\label{208link}
\end{figure}

By Corollaries \ref{maincor} and \ref{conjucor}, it is enough to show that the holonomy variety is invariant under the monomial transformations arise from the symmetries given in \eqref{25050103}. 

From the gluing equations, we find that the defining polynomial of the geometric component of the holonomy variety $\mathcal{X}$ of $M$ is given as follows:
\begin{align*}
    \begin{pmatrix}
    m^{-4} & \cdots & m^4
    \end{pmatrix}\begin{pmatrix}
     &  &  & & & & 1 & & \\
     &  &  &  & -2 & -2 & -2 & -2 &  \\
    & & 1 & -2 & 29 & -28 & 29 & -2 & 1  \\
    & & -2 & -28 & 2 & 2 & -28 & -2 &   \\
    & -2 & 29 & 2 & 12 & 2 & 29 & -2 & \\
     & -2 & -28 & 2 & 2 & -28 & -2 & & \\
    1 & -2 & 29 & -28 & 29 & -2 & 1 & & \\
   & -2 & -2 & -2 & -2 & & & & \\
     &  & 1 & &  & &  & & 
    \end{pmatrix}\begin{pmatrix}
    l^{-4}  \\
    \vdots \\
    l^4
    \end{pmatrix}.
\end{align*}

The coefficient matrix is that of Example 10.0.22 of \cite{C}. The symmetry of the polynomial shows that it is invariant under 
\begin{equation*}
m\mapsto l, l\mapsto m^{-1}l \quad \text{and}\quad m\mapsto l, l\mapsto m. 
\end{equation*}
By Corollaries \ref{maincor} and \ref{conjucor}, the conclusion follows. 
\end{exm}

\begin{exm}\label{135comp}
Let $M$ be the manifold m135 which is the punctured torus bundle over $S^1$ with monodromy $-LLRR$ (\cite{D}, Section 5). It has the cusp shape $c_1=i$. We claim that
\begin{align*}
    \textnormal{vol}_{\mathbb{C}}\,M(p,q)=\textnormal{vol}_{\mathbb{C}}\,M(-q,p)\quad&\mod  i\pi^2\mathbb{Z},\\
    \textnormal{vol}_{\mathbb{C}}\,M-\textnormal{vol}_{\mathbb{C}}\,M(p,q)=\overline{\textnormal{vol}_{\mathbb{C}}\,M-\textnormal{vol}_{\mathbb{C}}\,M(q,p)}\quad&\mod  i\pi^2\mathbb{Z}
\end{align*}
for sufficiently large $|p|+|q|$. 

\begin{figure}[htb]
\begin{tikzpicture}[line width=1.2, line cap=round, line join=round, scale=.6]
\definecolor{linkcolor0}{RGB}{235, 144, 130}
\definecolor{linkcolor1}{RGB}{111, 128, 221}
\definecolor{linkcolor2}{RGB}{209, 133, 189}
  \begin{scope}[color=blue]
    \draw (3.78, 4.90) .. controls (3.78, 5.60) and (3.45, 6.27) .. 
          (2.83, 6.57) .. controls (2.24, 6.85) and (1.55, 6.80) .. 
          (1.02, 6.42) .. controls (0.17, 5.83) and (0.15, 4.68) .. 
          (0.14, 3.64) .. controls (0.13, 2.57) and (0.17, 1.42) .. 
          (0.99, 0.75) .. controls (1.51, 0.34) and (2.20, 0.19) .. 
          (2.81, 0.46) .. controls (3.31, 0.68) and (3.74, 1.07) .. (3.75, 1.60);
    \draw (3.75, 2.01) .. controls (3.75, 2.32) and (3.76, 2.62) .. (3.76, 2.92);
    \draw (3.76, 2.92) .. controls (3.76, 3.08) and (3.76, 3.25) .. (3.76, 3.41);
    \draw (3.77, 3.82) .. controls (3.77, 4.18) and (3.77, 4.54) .. (3.78, 4.90);
  \end{scope}
  \begin{scope}[color=yellow!50!orange]
    \draw (6.23, 1.86) .. controls (6.24, 1.20) and (6.71, 0.67) .. 
          (7.31, 0.40) .. controls (7.91, 0.14) and (8.61, 0.14) .. 
          (9.09, 0.57) .. controls (9.81, 1.21) and (9.79, 2.37) .. 
          (9.77, 3.39) .. controls (9.74, 4.44) and (9.72, 5.57) .. 
          (8.97, 6.29) .. controls (8.49, 6.74) and (7.77, 6.81) .. 
          (7.16, 6.54) .. controls (6.59, 6.28) and (6.12, 5.79) .. (6.14, 5.17);
    \draw (6.15, 4.76) .. controls (6.16, 4.49) and (6.16, 4.22) .. (6.17, 3.94);
    \draw (6.18, 3.53) .. controls (6.19, 3.37) and (6.19, 3.20) .. (6.19, 3.03);
    \draw (6.19, 3.03) .. controls (6.21, 2.64) and (6.22, 2.25) .. (6.23, 1.86);
  \end{scope}
  \begin{scope}[color=magenta]
    \draw (3.77, 3.61) .. controls (3.07, 3.58) and (2.28, 3.60) .. 
          (2.27, 4.20) .. controls (2.26, 4.74) and (2.94, 4.88) .. (3.57, 4.90);
    \draw (3.98, 4.91) .. controls (4.70, 4.93) and (5.42, 4.95) .. (6.14, 4.97);
    \draw (6.14, 4.97) .. controls (7.10, 5.00) and (7.97, 4.36) .. 
          (7.97, 3.46) .. controls (7.98, 2.59) and (7.29, 1.88) .. (6.43, 1.86);
    \draw (6.02, 1.85) .. controls (5.26, 1.84) and (4.51, 1.82) .. (3.75, 1.81);
    \draw (3.75, 1.81) .. controls (3.10, 1.80) and (2.35, 1.78) .. 
          (2.30, 2.32) .. controls (2.26, 2.83) and (2.95, 2.88) .. (3.55, 2.91);
    \draw (3.97, 2.93) .. controls (4.64, 2.96) and (5.31, 2.99) .. (5.99, 3.02);
    \draw (6.40, 3.04) .. controls (6.59, 3.05) and (6.74, 3.22) .. 
          (6.72, 3.41) .. controls (6.70, 3.64) and (6.43, 3.75) .. (6.18, 3.74);
    \draw (6.18, 3.74) .. controls (5.37, 3.70) and (4.57, 3.66) .. (3.77, 3.61);
  \end{scope}
  \node at (5,1.2) {$(2,1)$};
  \node[right] at (9.8,3) {$(3,-1)$};
\end{tikzpicture}
\caption{The manifold m135 which is obtained from the exterior of the link $8_9^3$ by filling two cusps}\label{135link}
\end{figure}

The defining polynomial of the geometric component of the holonomy variety $\mathcal{X}$ of $M$ is given as
\begin{align*}
    \begin{pmatrix}
    m^{-1} & 1 & m
    \end{pmatrix}\begin{pmatrix}
    1 & 2 & 1 \\
    2 & -12 & 2 \\
    1 & 2 & 1
    \end{pmatrix}\begin{pmatrix}
    l^{-1}  \\
    1  \\
    l
    \end{pmatrix}, 
\end{align*}
and it is invariant under
\begin{equation*}
m\mapsto l^{-1}, l\mapsto m \quad \text{and}\quad m\mapsto l, l\mapsto m. 
\end{equation*}
As in the previous example, the claim is derived from Corollaries \ref{maincor} and \ref{conjucor}. 
\end{exm}

\begin{exm}\label{009comp}
Let $M$ be the manifold m009 which is the punctured torus bundle over $S^1$ with monodromy $LLR$ (\cite{DF} 2.6.2) and has the cusp shape $c_1=\sqrt{-7}$. The defining polynomial of the geometric component of the holonomy variety $\mathcal{X}$ of $M$ is 
\begin{align*}
    \begin{pmatrix}
    m^{-3} & \cdots & m^3
    \end{pmatrix}\begin{pmatrix}
     & -1 &  \\
     & 4 &  \\
     &  & \\
    1 & -8 & 1 \\
     &  &  \\
     & 4 & \\
     & -1 &
    \end{pmatrix}\begin{pmatrix}
    l^{-1}  \\
    1  \\
    l
    \end{pmatrix}, 
\end{align*}
which is invariant under $m\mapsto m, l\mapsto l^{-1}$. Hence
\begin{align*}
\textnormal{vol}_{\mathbb{C}}\,M-\textnormal{vol}_{\mathbb{C}}\,M(p,q)=\overline{\textnormal{vol}_{\mathbb{C}}\,M-\textnormal{vol}_{\mathbb{C}}\,M(-p,q)}\quad\mod i\pi^2\mathbb{Z}
\end{align*}
for sufficiently large $|p|+|q|$ follows by Corollary \ref{conjucor}.
\begin{figure}[htb]
\begin{tikzpicture}[line width=1.2, line cap=round, line join=round, scale=.6]
    \begin{scope}[color=red]
    \draw (2.42, 4.75) .. controls (2.81, 5.23) and (3.32, 5.61) .. 
          (3.90, 5.85) .. controls (4.49, 6.10) and (5.16, 6.07) .. 
          (5.76, 5.84) .. controls (6.29, 5.63) and (6.82, 5.38) .. (7.21, 4.96);
    \draw (7.45, 4.71) .. controls (7.86, 4.28) and (7.96, 3.66) .. 
          (7.93, 3.07) .. controls (7.90, 2.57) and (7.87, 2.00) .. (7.59, 1.63);
    \draw (7.59, 1.63) .. controls (7.19, 1.08) and (6.66, 0.62) .. 
          (6.02, 0.37) .. controls (5.43, 0.14) and (4.78, 0.11) .. 
          (4.17, 0.27) .. controls (3.48, 0.45) and (2.88, 0.88) .. (2.40, 1.41);
    \draw (2.17, 1.66) .. controls (1.84, 2.03) and (1.83, 2.55) .. 
          (1.84, 3.04) .. controls (1.84, 3.66) and (2.03, 4.27) .. (2.42, 4.75);
\end{scope}
  \begin{scope}[color=orange]
      \draw (2.29, 1.53) .. controls (1.55, 1.40) and (0.77, 1.60) .. 
          (0.43, 2.23) .. controls (0.12, 2.81) and (0.12, 3.50) .. 
          (0.47, 4.04) .. controls (0.84, 4.62) and (1.55, 4.84) .. (2.24, 4.77);
    \draw (2.59, 4.73) .. controls (3.38, 4.64) and (4.24, 3.86) .. (4.95, 3.21);
    \draw (4.95, 3.21) .. controls (5.68, 2.54) and (6.45, 1.86) .. (7.42, 1.66);
    \draw (7.77, 1.60) .. controls (8.39, 1.47) and (9.02, 1.76) .. 
          (9.39, 2.28) .. controls (9.78, 2.82) and (9.82, 3.52) .. 
          (9.50, 4.09) .. controls (9.09, 4.81) and (8.18, 4.93) .. (7.33, 4.84);
    \draw (7.33, 4.84) .. controls (6.45, 4.74) and (5.73, 4.00) .. (5.07, 3.33);
    \draw (4.83, 3.08) .. controls (4.11, 2.36) and (3.28, 1.72) .. (2.29, 1.53);
  \end{scope}
  \node[right] at (8.5,1) {$(4,-1)$};
\end{tikzpicture}
\caption{The manifold m009 which is obtained from the Whitehead link complement by filling one cusp}\label{009link}
\end{figure}
\end{exm}


\begin{thebibliography}{FGGTV}
\bibitem[Ax]{A} J. Ax, \emph{On Schanuel’s conjectures}, Ann. of Math. (2) \textbf{93} (1971), 252–268.

\bibitem[Bak]{B} A. Baker, \emph{Transcendental number theory}, Cambridge Univ. Press, Cambridge, 1990.

\bibitem[BPZ]{BPZ} S. Betley, J. H. Przytycki, and T. Zukowski, \emph{Hyperbolic structures on Dehn fillings of some punctured-torus bundles over $S^1$}, Kobe J. Math. \textbf{3} (1987), 117-147.

\bibitem[BD]{BD} N. Bhardwaj and L. van den Dries, \emph{On the Pila–Wilkie theorem}, Expo. Math. \textbf{40} (2022), 495–542.

\bibitem[BP]{BP} E. Bombieri and J. Pila, \emph{The number of integral points on arcs and ovals}, Duke Math. J. \textbf{59} (1989), 337-357.

\bibitem[Bur]{Bur} B. A. Burton, \emph{The cusped hyperbolic census is complete}, \href{https://arxiv.org/abs/1405.2695}{arXiv:1405.2695}.

\bibitem[CHW]{CHW} P. J. Callahan, M. V. Hildebrand, and J. R. Weeks, \emph{A census of cusped hyperbolic 3-manifolds}, Math. Comp. \textbf{68} (1999), 321–332.

\bibitem[Cha]{C} A. A. Champanerkar, \emph{A-polynomial and Bloch invariants of hyperbolic 3-manifolds}, ProQuest LLC, Ann Arbor, MI, 2003, Thesis (Ph.D.)–Columbia University. 

\bibitem[CGHN]{CGHN} D. Coulson, O. Goodman, C. Hodgson, and W. Neumann, \emph{Computing arithmetic invariants of 3-manifolds}, Exp. Math. \textbf{9} (2000), 127–152.

\bibitem[CCGLS]{CCGLS} D. Copper, M. Culler, H. Gillett, D. Long and P. Shalen, \emph{Plane curves associated to character varieties of knot complements}, Invent. Math. \textbf{118} (1994), 47-84. 

\bibitem[CDGW]{CDGW} M. Culler, N. M. Dunfield, M. Goerner, and J. R. Weeks, ``SnapPy, a computer program for studying the geometry
and topology of 3-manifolds”, \url{http://snappy.computop.org}.

\bibitem[DF]{DF} R. Dijkgraaf and H. Fuji, \emph{The volume conjecture and topological strings}, Fortsch. Phys. \textbf{57} (2009), 825-856.

\bibitem[DS]{DS} J.L. Dupont and H. Sah, \emph{Scissors congruences II}, J. Pure and App. Algebra \textbf{25} (1982), 159-195.

\bibitem[Dun]{D} N. M. Dunfield, \emph{Examples of non-trivial roots of unity at ideal points of hyperbolic 3-manifolds}, Topology \textbf{38} (1999), no. 2, 457–465.

\bibitem[EMSW]{EMSW} K. Eisenträger, R. Miller, C. Springer, and L. Westrick, \emph{A topological approach to undefinability in algebraic extensions of $\mathbb{Q}$}, Bull. Symb. Logic, \textbf{29} (2023), no. 4, 626-655. 

\bibitem[FGGTV]{FGGTV} E. Fominykh, S. Garoufalidis, M. Goerner, V. Tarkaev, and A. Vesnin, \emph{A Census of Tetrahedral Hyperbolic Manifolds}, Exp. Math. \textbf{25} (2016), no. 4, 466–481. 

\bibitem[FPS]{FPS} D. Futer, J. S. Purcell, and S. Schleimer, \emph{Excluding cosmetic surgeries on hyperbolic 3-manifolds}, \href{https://arxiv.org/abs/2403.10448}{arXiv:2403.10448}.

\bibitem[Gro]{G} M. Gromov, \emph{Hyperbolic manifolds (according to Thurston and Jørgensen)}, Bourbaki Seminar (1979/80), 40-53, Lecture Notes in Math. \textbf{842}, Springer, Berlin-New York, 1981.

\bibitem[HW]{HW} M. V. Hildebrand and J. R. Weeks, \emph{A computer generated census of cusped hyperbolic 3-manifolds},
Computers and Mathematics (Cambridge, MA, 1989) Springer, New York, 1989, 53–59. 

\bibitem[HM]{HM} C. Hodgson and H. Masai, \emph{On the number of hyperbolic 3-manifolds of a given volume}, Contemp. Math. \textbf{597} (2013), 295–320.

\bibitem[HMW]{HMW} Craig D. Hodgson, G. Robert Meyerhoff, and Jeffrey R. Weeks, \emph{Surgeries on the Whitehead link yield geometrically similar manifolds}, Topology ’90 (Columbus, OH, 1990) Ohio State Univ. Math. Res. Inst. Publ., vol. 1, de Gruyter, Berlin, 1992, 195–206.

\bibitem[Jeo1]{J1} B. Jeon, \emph{Classification of hyperbolic Dehn fillings I}, Proc. of Lond. Math. Soc. \textbf{130} (2025), e70017.

\bibitem[Jeo2]{J2} B. Jeon, \emph{On the number of hyperbolic Dehn fillings of a given volume}, Trans. Amer. Math. Soc. \textbf{374} (2021), 3947–3969.

\bibitem[JS]{JS} B. Jeon and S. Oh, \emph{Code for searching Dehn fillings of the same volume on a given hyperbolic 3-manifold}, \url{https://github.com/SunulOh/equal-volume-dehn-fillings}.

\bibitem[Neu]{N} W. D. Neumann, \emph{Combinatorics of triangulations and the Chern-Simons invariant for hyperbolic 3-manifolds}, Topology ’90 (Columbus, OH, 1990) Ohio State Univ. Math. Res. Inst. Publ., vol. 1, de Gruyter, Berlin, 1992, 243–271.

\bibitem[NR]{NR} W. D. Neumann and A. W. Reid, \emph{Arithmetic of hyperbolic manifolds}, Topology ’90 (Columbus, OH, 1990) Ohio State Univ. Math. Res. Inst. Publ., vol. 1, de Gruyter, Berlin, 1992, 273–310.

\bibitem[NY]{NY} W. D. Neumann and J. Yang, \emph{Bloch invariants of hyperbolic 3-manifolds}, Duke Math. J. \textbf{96} (1999), 29-59. 

\bibitem[NZ]{NZ} W. D. Neumann and D. Zagier, \emph{Volumes of hyperbolic three-manifolds}, Topology \textbf{24} (1985), no. 3, 307–332.

\bibitem[Ser]{S} J.-P. Serre, \emph{Topics in Galois theory}, second ed., vol. 1 of Research Notes in Mathematics. A K Peters,
Ltd., Wellesley, MA, 2008. With notes by Henri Darmon.

\bibitem[Thi]{Thi} M. Thistlethwaite, ``Cusped hyperbolic 3-manifolds constructed from 8 ideal tetrahedra", 2010, \url{https://www.math.utk.edu/~morwen/8tet}.

\bibitem[Thu]{T2} W. P. Thurston, \emph{Three dimensional manifolds, Kleinian groups and hyperbolic geometry}, Bull. Amer. Math. Soc. \textbf{6} (1982), 357-381. 

\bibitem[Yos]{Y} T. Yoshida, \emph{The $\eta$-invariant of hyperbolic 3-manifolds}, Invent. Math. \textbf{81} (1985), 473–514.
\end{thebibliography}
\end{document}